\documentclass[draft]{siamltex}
\usepackage{amssymb}\usepackage{amsmath}
\usepackage[all, knot]{xy}
\xyoption{arc} \xyoption{color}
\usepackage{mathrsfs}

\newtheorem{remark}[theorem]{Remark}

\title{Limiting absorption principle and perfectly matched layer method for  Dirichlet Laplacians in quasi-cylindrical domains.  \thanks{This
        work was supported by grant number 108898 awarded by the Academy
of Finland.}}
\author{Victor Kalvin \thanks{Department of Mathematics and Statistics, Concordia University, 1455 de
Maisonneuve West, Montreal, H3G 1M8 Quebec, Canada
 (\tt vkalvin@gmail.com)}}
\begin{document}
\maketitle
\begin{abstract}
We establish a limiting absorption principle for Dirichlet
Laplacians in quasi-cylindrical domains. Outside a bounded set these
domains can be transformed onto a semi-cylinder by suitable
diffeomorphisms. Dirichlet Laplacians model quantum or
acoustically-soft waveguides associated with quasi-cylindrical
domains.  We construct a uniquely solvable problem with perfectly
matched layers of finite length. We prove that solutions of the
latter problem approximate outgoing or incoming solutions with an
error that exponentially tends to zero as the length of layers tends
to infinity. Outgoing and incoming solutions are characterized by
means of the limiting absorption principle.
\end{abstract}



\pagestyle{myheadings} \thispagestyle{plain} \markboth{VICTOR
KALVIN}{DIRICHLET LAPLACIANS IN QUASI-CYLINDRICAL DOMAINS}


\section{Introduction}
 The perfectly matched layer (PML) method,
originally introduced in~\cite{ref1}, is in common use for the
numerical analysis of a wide class of problems. For some of them
 stability and convergence of the method have been proved
mathematically; see,
e.g.,~\cite{ref2,ref3,ref4,ref4+,ref5,ref5+,KalvinSiNum}. In the
present paper we develop the PML method for  Dirichlet Laplacians in
 quasi-cylindrical domains $\mathcal G\subset\mathbb R^{n+1}$, see
Fig.~\ref{fig1}. These are unbounded domains that outside a bounded
set can be transformed onto a semi-cylinder by suitable
diffeomorphisms.
 \begin{figure}[h]
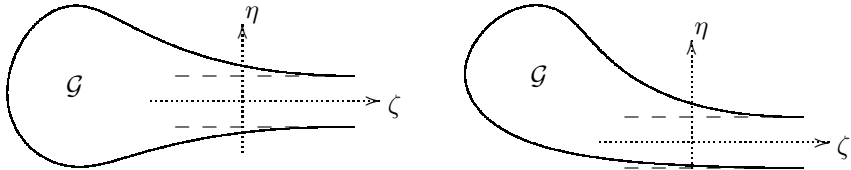

\[\xy (0,0)*{\xy0;/r.16pc/:
  (35,5); (35,-5) **\crv{(0,5)&(-20,30)&(-40,0)&(-20,-20)&(0,-5)};
 {\ar@{.>} (-5,0)*{}; (40,0)*{}}; (43,-1)*{\zeta}; {\ar@{.>}
 (13,-10)*{}; (13,15)*{}};(15,17)*{\eta}; (-20,3)*{{\mathcal G}};
 (0,5)*{};(35,5)*{}**\dir{--};(0,-5)*{};(35,-5)*{}**\dir{--};\endxy};
(60,0)*{\xy0;/r.16pc/: (35,5); (35,-5)
**\crv{(0,5)&(-10,35)&(-35,20)&(-30,0)&(0,-5)}; {\ar@{.>}
(-5,0)*{}; (40,0)*{}}; (43,-1)*{\zeta}; {\ar@{.>} (13,-5)*{};
(13,20)*{}}; (15,22)*{\eta};(-17,13)*{{\mathcal G}};
(0,5)*{};(35,5)*{}**\dir{--};(0,-5)*{};(35,-5)*{}**\dir{--};\endxy};
\endxy
\]
\caption{Examples of quasi-cylindrical domains in $\Bbb
R^2$.}\label{fig1}
\end{figure}
Intuitively, one can understand $\mathcal G$ as a domain whose
 boundary asymptotically approaches at infinity the boundary
of a semi-cylinder $(0,\infty)\times\Omega$, where the cross-section
$\Omega$ of $\mathcal G$ at infinity is a bounded domain in $\Bbb
R^n$. Dirichlet Laplacians $\Delta$ model quantum or
acoustically-soft waveguides associated with quasi-cylindrical
domains. In order to characterize outgoing and incoming solutions of
the Helmholtz equation $(\Delta-\mu)u=f$ we establish a limiting
absorption principle. Then we construct a uniquely solvable problem
with PMLs of finite length. This is a Dirichlet problem  in the
domain $\mathcal G$ truncated at a finite distance. We prove that
solutions of the latter problem locally approximate outgoing or
incoming solutions of the Helmholtz equation with an error that
exponentially tends to zero as the length of PMLs tends to infinity.
In other words, we prove stability and exponential convergence of
the PML method. We find that the rate of exponential convergence
depends only on the spectral parameter $\mu$ and on the infinitely
distant cross-section $\Omega$. Thus the rate is the same as in the
particular case of a domain $\mathcal G$ that coincides with the
semi-cylinder $(0,\infty)\times\Omega$ outside a bounded set.

As is known,  construction of PMLs is closely related to the complex
scaling. The complex scaling involves complex dilation of variables
and has a long tradition in mathematical physics and numerical
analysis; for a historical account see e.g.~\cite{Hislop
Sigal,Cycon,Simon Reed iv}.  Although there are several papers
utilizing different approaches to the complex scaling in
waveguide-type of geometry, e.g.~\cite{arsenev,DES,DEM,KalvinSiNum},
the complex scaling has not been used in quasi-cylindrical domains
before. Our approach  to the complex scaling originates from the one
developed in~\cite{Hunziker} for a Schr\"{o}dinger operator in $\Bbb
R^3$ (see also~\cite{Hislop Sigal}). Deformations of the Dirichlet
Laplacian by means of the complex scaling give rise to an analytic
family of non-selfadjoint operators in $\mathcal G$.   These
operators correspond to a Dirichlet problem with infinite PMLs.
Localization of the essential spectra of these operators together
with  certain relations between their resolvents  justifies a
limiting absorption principle. For locating the essential spectra we
employ methods of the theory of elliptic boundary value
problems~\cite{KozlovMaz`ya,KozlovMazyaRossmann,MP2}. Relations
between the resolvents are obtained with the help of Hardy spaces of
analytic functions. Note that Hardy spaces in context of the complex
scaling were originally used in~\cite{Van Winter I,Van Winter}. As
we mention in Remark~\ref{ABCS}, our methods also make it possible
to develop an analog of the celebrated Aguilar-Balslev-Combes-Simon
theory of resonances.

As is typically the case, solutions to the Helmholtz equation
satisfying the limiting absorption principle locally coincide with
solutions to the problem with infinite PMLs. Moreover, under certain
assumptions on the right hand side solutions to the latter problem
are of some exponential decay at infinity. This allows us to prove
unique solvability of the problem with finite PMLs and establish
exponential convergence of the PML method by using the compound
expansion technique~\cite{ref6,KozlovMazyaRossmann}.
In~\cite{KalvinSiNum} we used a similar approach to study the PML
method for inhomogeneous media. Here we study the PML method for a
wide class of quasi-cylindrical domains.


This paper is organized as follows. Section~\ref{s1} consists of
preliminaries, where we introduce notations, formulate our
assumptions on the quasi-cylindrical domains, and give a formal
definition of operators corresponding to the problem with infinite
PMLs. 
In Section~\ref{sec3} we demonstrate that the operators are
well-defined and derive some estimates on their coefficients.
As shown in Section~\ref{secAF}, these operators give rise to an analytic family of m-sectorial operators. 
In Section~\ref{s3} we introduce and study Hardy spaces of analytic
functions. In Section~\ref{s6} we formulate and prove a limiting
absorption principle. In Section~\ref{s5} we show that under certain
assumptions on the right hand side solutions to the problem with
infinite PMLs are of some exponential decay at infinity. Finally, in
Section~\ref{secPML} we formulate and study the problem with finite
PMLs and prove exponential convergence of the PML method.

\section{Preliminaries} \label{s1} In this section we introduce basic notations that are in use throughout the paper. We formulate
our assumptions on the quasi-cylindrical domains and introduce
differential operators corresponding to the problem with infinite
PMLs. Recall that PMLs are artificial strongly absorbing layers
designed so that waves coming from a non-PML medium to PMLs do not
reflect at the interface.

Let $(x,y)$ and $(\zeta,\eta)$ be two systems of the Cartesian
coordinates in $\mathbb R^{n+1}$, $n\geq 1$, such that
$x,\zeta\in\mathbb R$, while $y=(y_1,\dots,y_n)$ and
$\eta=(\eta_1,\dots,\eta_n)$ are in $\mathbb R^n$. Let
$\partial_x=\frac d {d x}$, $\partial_{y_m}=\frac d {d y_{m}}$,
  and $\partial_\zeta=\frac d {d \zeta}$, $\partial_{\eta_m}=\frac d {d \eta_{m}}$.

Consider a  bounded  domain $\Omega\subset\mathbb R^n$, and the
semi-cylinder $\mathbb R_+\times\overline{\Omega}$, where $\mathbb
R_+=\{x\in\mathbb R: x>0\}$. We say that $\mathcal C\subset\mathbb
R^{n+1}$ is a quasi-cylinder, if there exists a diffeomorphism
\begin{equation}\label{diff}
\mathbb R_+\times\overline{\Omega}\ni(x,y)\mapsto
\varkappa(x,y)=(\zeta,\eta)\in \overline{\mathcal C},
\end{equation}
such that the elements $\varkappa'_{\ell m}(x,\cdot)$ of its
Jacobian matrix $\varkappa'$ tend to the Kronecker delta
$\delta_{\ell m}$ in the space $C^\infty(\overline{\Omega})$  as
$x\to+\infty$.

 Let
${\mathcal G}$ be a  domain in $\mathbb R^{n+1}$ with smooth
boundary $\partial {\mathcal G}$. We suppose that the set
$\{(\zeta,\eta)\in {\mathcal G}: \zeta\leq 0\}$ is bounded, and the
set $\{(\zeta,\eta)\in{\mathcal G}:\zeta> 0\}$ coincides with a
quasi-cylinder $\mathcal C$. (Extension of our results to the case
of a domain $\mathcal G$ that coincides outside a bounded set with
several quasi-cylinders is straightforward.)
Following~\cite{KozlovMazyaRossmann}, we say that $\mathcal G$ is a
quasi-cylindrical domain.

 Introduce the notation
$\nabla_{\zeta\eta}=(\partial_\zeta,\partial_{\eta_1},\dots,\partial_{\eta_n})^\top$.
In the domain $\mathcal G$ we consider the Dirichlet Laplacian
$\Delta=-\nabla_{\zeta\eta}\cdot\nabla_{\zeta\eta}$, which  is
initially defined on the set $C_0^\infty(\overline{\mathcal G})$ of
all smooth compactly supported
 functions $u$ in $\overline{\mathcal G}$ satisfying the
Dirichlet boundary condition $u\upharpoonright_{\partial\mathcal
G}=0$.

Consider the complex scaling $x\mapsto x+\lambda s(x-r)$ with
parameters $r>0$ and $\lambda\in\Bbb C$. Here $s(x)$ is a smooth
scaling function possessing the properties:
\begin{eqnarray}
& s(x)=0 \text{ for } x\leq 0,\label{ab1}\\
& 0\leq s'(x)\leq 1 \text{ for all } x\in\mathbb R,\label{ab2}\\
& s'(x)=1 \text{ for } x\geq C>0\label{ab3},
\end{eqnarray}
where $s'(x)=\partial_x s(x)$, and $C$ is arbitrary. For all real
$\lambda\in(-1,1)$ the function $\mathbb R_+\ni x\mapsto x+\lambda
s(x-r)$ is invertible, and
$$
\Bbb R_+\times\overline{\Omega}\ni(x,y)\mapsto\kappa_{\lambda,r}
(x,y)=(x+\lambda s(x-r),y)\in\Bbb R_+\times\overline{\Omega}
$$
is a selfdiffeomorphism of the semi-cylinder. Therefore
\begin{equation}\label{sc}
\vartheta_{\lambda,r}(\zeta,\eta)=\left\{
                      \begin{array}{lll}
                        \varkappa\circ\kappa_{\lambda,r}\circ\varkappa^{-1}(\zeta,\eta)&\text{ for }  & (\zeta,\eta)\in\overline{\mathcal C}, \\
                       (\zeta,\eta) &\text{ for }& (\zeta,\eta)\in\overline{\mathcal G}\setminus\overline{\mathcal C},
                      \end{array}
                    \right.
\end{equation}
is a selfdiffeomorphism of   $\overline{\mathcal G}$. In other
words, $\vartheta_{\lambda,r}$ with $\lambda\in(-1,1)$ and $r>0$ is
a scaling of the quasi-cylinder $\mathcal C$ along the curvilinear
coordinate $x$.
 Let $(\vartheta_{\lambda,r}')^\top$ be the transpose of the Jacobian matrix $\vartheta_{\lambda,r}'$. Then
$\mathsf e_{\lambda,r}=(\vartheta_{\lambda,r}')^\top
\vartheta_{\lambda,r}'$ is the matrix coordinate representation of a
 metric $\mathsf e_{\lambda,r}$ on $\overline{\mathcal
G}$, and
\begin{equation}\label{v2}
{\Delta_{\lambda,r}}=-\bigl({\det \mathsf e_{\lambda,r}
}\bigr)^{-1/2}\nabla_{\zeta\eta}\cdot\bigl({\det \mathsf
e_{\lambda,r} }\bigr)^{1/2}\mathsf
e^{-1}_{\lambda,r}\nabla_{\zeta\eta}
\end{equation}
is the Laplace-Beltrami operator on the Riemannian manifold
$(\overline{\mathcal G},\mathsf e_{\lambda,r})$.
  As the parameter $r$  increases,  the equalities $\vartheta_{\lambda,r}(\zeta,\eta)=(\zeta,\eta)$,  $\mathsf
e_{\lambda,r}(\zeta,\eta)=\operatorname{Id}$, where
$\operatorname{Id}$ is the $(n+1)\times(n+1)$ identity matrix, and
the equality $\Delta_{\lambda,r}=\Delta$  become valid on a larger
and larger subset of $\mathcal G$.  In the case $\lambda=0$ the
scaling is not applied.  Therefore $\mathsf e_{0, r}\equiv\mathsf e$
is the Euclidean metric  and $\Delta_{0,r}\equiv\Delta$.

 In order to  consider  complex values of the scaling parameter $\lambda$, we impose  additional assumptions on the diffeomorphism~\eqref{diff}:
\begin{enumerate}
\item[\it i.] the
function $\mathbb R_+\ni x\mapsto \varkappa(x,\cdot)\in
C^\infty(\overline{\Omega})$ has an analytic continuation from
$\mathbb R_+$ to some sector
    \begin{equation}\label{S}
    \mathbb S_\alpha=\{z\in\mathbb C: |\arg
    z|<\alpha<\pi/4\};
    \end{equation}
\item[\it ii.] the elements $\varkappa'_{\ell m}(z,\cdot)$ of the Jacobian matrix $\varkappa'$ tend to the
Kronecker delta $\delta_{\ell m}$ in the space
$C^\infty(\overline{\Omega})$ uniformly in $z\in\mathbb S_\alpha$
as $z\to \infty$.
\end{enumerate}

For instance, the assumptions~{\it i,ii}  are satisfied for the
following quasi-cylinders:
$$
 \mathcal C=\{(\zeta,\eta)\in\mathbb R^2:(\zeta,\eta)=(x,y+\log(x+2)),\ x\in \mathbb R_+,\ y\in[0,1]\},
$$
$$
 \mathcal C=\Bigl\{(\zeta,\eta)\in \mathbb R^{2}: \zeta=\int_0^x\varphi(t)dt,\
 \eta=y\psi(x),\ x\in \mathbb R_+,\ y\in[0,1]\Bigr\},
$$
where as $\varphi(x)$ and $\psi(x)$ we can  take the functions $1$,
$1+e^{-x}$, $1+(x+1)^{-s}$ with $s>0$,  $1+1/\log(x+2)$,
$1+1/\log(1+\log(x+2))$, and so on. These examples show that
quasi-cylinders can have very different shapes comparing with the
semi-cylinder.

In the next section we will show that for all sufficiently large
$r>0$ the assumptions {\it i, ii} on $\varkappa$ together
with~\eqref{ab1} and \eqref{ab2} lead to the analyticity of the
coefficients of the differential operator~\eqref{v2} with respect to
the scaling parameter $\lambda$ in the disk
\begin{equation}\label{disk }
\mathcal D_\alpha=\{\lambda\in\mathbb C:
|\lambda|<\sin\alpha<1/\sqrt{2}\}.
\end{equation}
Thus the equality~\eqref{v2} defines $\Delta_{\lambda,r}$ for all
$\lambda\in\mathcal D_\alpha$. Clearly, $\Delta_{\lambda,r}$
coincides with the Dirichlet Laplacian $\Delta$ on the set
\begin{equation}\label{truncated}
\mathcal G_r=(\mathcal G\setminus\mathcal C)\cup\{(\zeta,\eta)\in
\mathcal C: (\zeta,\eta)=\varkappa(x,y), x< r, y\in\Omega\}.
\end{equation}
We will show that on the set $\mathcal G\setminus\mathcal G_{r}$ the
operator $\Delta_{\lambda,r}$ with $\lambda\in\mathcal
D_\alpha\setminus\Bbb R$ describes an infinite PML  for $\Delta$. In
the case $\Im\lambda>0$ (resp. $\Im\lambda<0$) this PML is an
artificial nonreflective strongly absorbing layer for the outgoing
(resp. incoming) solutions.

\begin{remark} For simplicity we consider in this paper  only Dirichlet Laplacians.
However, similar methods can be used to develop and study PML method
for the  Schr\"{o}dinger operator  $\Delta+V$ in $\mathcal G$, where
$\Delta$ is the Dirichlet Laplacian and $V\in
C^\infty(\overline{\mathcal G})$ is a real-valued 
potential with the following properties: for some $r_0>0$
and $\alpha>0$ the function $x\mapsto V\circ\varkappa(x,\cdot)\in
L^2(\Omega)$ extends by analyticity to the sector $\{z\in \Bbb C:
|\arg(z-r_0)|<\alpha\}$, where for all $y\in\overline{\Omega}$ we
have $|V\circ\varkappa(z,y)|\leq C(|z|)\to 0$  as $z\to\infty$. One
can also include into consideration potentials with moderate local
singularities and relatively bounded operator-valued potentials.
\end{remark}

\section{Construction of infinite PMLs}\label{sec3}
In this section we show that for all sufficiently large $r>0$ the
differential operator~\eqref{v2} is well-defined for complex values
of the scaling parameter $\lambda$ in the disk~\eqref{disk }. We
also obtain some estimates on the matrix $\mathsf e_{\lambda, r}$.

Consider the quasi-cylinder $\overline{\mathcal C}$ as a manifold
endowed with the Euclidean metric $\mathsf e$. We will use the
coordinates $(\zeta,\eta)$ in $\overline{\mathcal G}$ and $(x,y)$ in
$\Bbb R_+\times\overline{\Omega}$, and identify the Riemannian
metrics on $\overline{\mathcal G}$ and $\Bbb
R_+\times\overline{\Omega}$ with their matrix coordinate
representations. Let $\mathsf g=\varkappa^* \mathsf e$ be the
pullback of the metric  $\mathsf e$ by the diffeomorphism
$\varkappa$ in~\eqref{diff}. Then the  matrix $\mathsf g=[\mathsf
g_{\ell m}]_{\ell ,m=1}^{n+1}$ is given by the equality $\mathsf
g=(\varkappa')^\top\varkappa'$, where $(\varkappa')^\top$ is the
transpose of the Jacobian $\varkappa'$. Since the diffeomorphism
$\varkappa$ satisfies the assumptions~{\it i,ii} of
Section~\ref{s1}, we conclude that  the metric matrix elements
\begin{equation}\label{partan}
\mathbb S_\alpha\ni z\mapsto  \mathsf g_{\ell m}(z,\cdot)\in
C^\infty(\overline{\Omega})
\end{equation}
are analytic functions. Moreover, $\mathsf g_{\ell m}(z,\cdot)$
tends to the Kronecker delta $\delta_{\ell m}$ in the space
$C^\infty(\overline{\Omega})$ uniformly in $z\in \mathbb S_\alpha$
as $z\to \infty$ or, equivalently, we have
 \begin{equation}\label{stab}
  \bigl|\partial_y^q(\mathsf g_{\ell m}(z,y)-
 \delta_{\ell m})\bigr |\leq C_q(|z|)\to 0  \text{ as } z\to \infty,\ z\in \mathbb S_\alpha,\ y\in\overline{\Omega},\ |q|\geq 0;
 \end{equation}
here
$\partial^q_{y}=\partial^{q_1}_{y_1}\partial^{q_2}_{y_2}\dots\partial^{q_n}_
{y_n}$ with a multiindex $q=(q_1,\dots,q_n)$, and  $|q|=\sum q_j$.

Consider the selfdiffeomorphism $\kappa_{\lambda,r}$,
$\lambda\in(-1,1)$, of the semi-cylinder $\Bbb
R_+\times\overline{\Omega}$.  We define the metric $\mathsf
g_{\lambda,r}={\kappa_{\lambda,r}}^*\mathsf g$ on $\Bbb
R_+\times\overline{\Omega}$ as the pullback of the metric~$\mathsf
g$ by $\kappa_{\lambda,r}$. As a result we get the manifold $(\Bbb
R_+\times\overline{\Omega}, \mathsf g_{\lambda,r})$ parameterized by
$\lambda\in(-1,1)$ and $r>0$. We deduce
 \begin{equation}\label{metric}
 \mathsf g_{\lambda,r}(x,y)=\diag\left\{1+\lambda
s'(x-r),\operatorname{Id} \right\}\mathsf g(x+\lambda
 s(x-r),y)\diag\left\{1+\lambda s'(x-r),\operatorname{Id} \right\},
 \end{equation}
 where $\operatorname{Id}$
 stands for the $n\times n$-identity matrix, and $\diag\left\{1+\lambda s'(x-r),\operatorname{Id} \right\}$ is
 the Jacobian of $\kappa_{\lambda,r}$.

Let us consider complex values of the scaling parameter $\lambda$.
We suppose that $\lambda$ is in the complex disk $\mathcal
D_\alpha$, where $\alpha$ is the same as in our assumptions on the
partial analytic regularity of the diffeomorphism $\varkappa$;
cf.~\eqref{S},~\eqref{disk }. The curve $\mathfrak
L_{\lambda,r}=\{z\in\mathbb C: z=x+\lambda s(x-r),x>0 \}$ lies in
the sector $\mathbb S_\alpha$, see Fig.~\ref{fig+}.
\begin{figure}
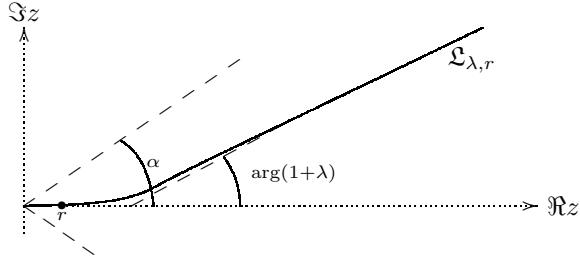
 \[ \xy0;/r.17pc/:
{\ar@{.>} (-12,0);(75,0)};(80,0)*{\Re z};
(-20,0);(20,27)**\dir{--};(-20,0);(-7,-9)**\dir{--};
(4,0);(-2,12)**\crv{(3,9)};(4,8)*{\scriptstyle\alpha};
{\ar@{.>} (-20,-5);(-20,33)}; (-20,36)*{ \Im z};
(-20,0);(65,33)**\crv{(-2,0)&(5,4)&(15,9)&(36,19)};
(0,0);(28,15)**\dir{--};(-13,0)*{\scriptstyle\bullet};(-13,-2)*{\scriptstyle
r};
(17,9);(20,0)**\crv{(20,5)};(30,6)*{\scriptstyle\arg(1+\lambda)};
(63,27)*{ \mathfrak L_{\lambda,r}};
\endxy\]
\caption{The curve $\mathfrak L_{\lambda,r}$  for complex values of
$\lambda$.}\label{fig+}
\end{figure}
We define the matrix $\mathsf g_{\lambda,r}$ for all non-real
$\lambda$ in the disk by the equality~\eqref{metric}, where $\mathsf
g(x+\lambda s(x-r),y)$ stands for the value of the analytic in
$z\in\mathbb S_\alpha$ function $\mathsf g(z,y)$ at $z=x+\lambda
s(x-r)$.
 By
 analyticity in $\lambda$ we conclude that  ${\mathsf g}_{\lambda,r}$  is a
 complex symmetric matrix, the Schwarz reflection principle gives $\overline{{\mathsf g}_{\lambda,r}}= {\mathsf g}_{\overline{\lambda},r}$,
 where the overline stands for the complex conjugation. If  $\lambda\in\mathcal D_\alpha$ is non-real, then the matrix  ${\mathsf g}_{\lambda,r}$
  does not correspond to  a Riemannian metric. However,
${\mathsf g}_{\lambda,r}$ is invertible for all $\lambda\in\mathcal
D_\alpha$ provided $r>0$ is sufficiently large. Indeed, $s(x-r)=0$
and $\mathsf g^{-1}_{\lambda,r}(x,y)= \mathsf g^{-1}(x,y)$ for all
$x<r$. On the other hand,   the matrix $\mathsf g(x+\lambda
s(x-r),y)$ from~\eqref{metric} is invertible for $x\geq r$ with
large $r>0$ as it is only little different from the identity matrix;
the latter fact is a consequence of~\eqref{stab} and the inequality
$|x+\lambda s(x-r)|\geq r$.

The derivatives $\partial_x^p\partial_y^q \mathsf
g^{-1}_{\lambda,r}$ are analytic
functions of $\lambda\in\mathcal D_\alpha$. 
From the conditions~\eqref{partan} and~\eqref{stab} 
together with~\eqref{ab3} and~\eqref{metric} it follows that
  \begin{equation}\label{lim}
\left \|\partial_x^p\partial_y^q\bigl(\mathsf
g^{-1}_{\lambda,r}(x,y) - \diag\bigl\{ (1+\lambda
)^{-2},\operatorname{Id} \bigr\}\bigr)\right\|\to 0 \text{ as }
x\to+\infty,\  y\in\overline{\Omega},\ \lambda\in\mathcal D_\alpha,
\end{equation}
and  the estimate
\begin{equation}\label{nstab}
\left \|\partial_x^p\partial_y^q\bigl(\mathsf
g^{-1}_{\lambda,r}(x,y) - \diag\bigl\{ (1+\lambda
s'(x-r))^{-2},\operatorname{Id} \bigr\}\bigr)\right\|\leq \mathrm{
C}_{pq}(r)
\end{equation}
holds uniformly in  $(x,y)\in[r,\infty)\times\overline{\Omega}$ and
$\lambda\in\mathcal D_\alpha$, where $\|\cdot\|$ is the matrix norm
$\|A\|=\max_{\ell m}|a_{\ell m}|$, and $p+|q|\geq 0$. The constants
$\mathrm{C}_{pq}(r)$ in~\eqref{nstab} tend to zero as $r\to+\infty$.

 We define the complex scaling $\vartheta_{\lambda,r}$ for all $\lambda$ in the disk $\mathcal D_\alpha$ by the equality~\eqref{sc},
 where $\varkappa\circ\kappa_{\lambda,r}(x,y)$  is  the value of the analytic in $z\in\mathbb S_\alpha$ function~$\varkappa(z,y)$ at the point $z=x+\lambda s(x-r)$.
  Consider the matrix $\mathsf e_{\lambda,r}=(\vartheta_{\lambda,r}')^\top \vartheta_{\lambda,r}'$. It is clear that for all  $(\zeta,\eta)\in\overline{\mathcal
  G}\setminus\overline{\mathcal C}$ the matrix $\mathsf e_{\lambda,r}(\zeta,\eta)$
coincides with the $(n+1)\times(n+1)$-identity. For all real
$\lambda\in\mathcal D_\alpha$  we have
\begin{equation}\label{rrr}
\mathsf
g_{\lambda,r}(x,y)=\bigl(\varkappa'(x,y)\bigr)^\top\bigl(\mathsf
e_{\lambda,r}\circ\varkappa(x,y)\bigr)\varkappa'(x,y),\quad
(x,y)\in\Bbb R_+\times\overline{\Omega},
\end{equation}
where  $\mathsf g_{\lambda,r}=\varkappa^*\mathsf e_{\lambda,r}$ is
the pullback of the corresponding metric $\mathsf e_{\lambda,r}$ on
$\mathcal C$ by the diffeomorphism $\varkappa$. Therefore $\mathsf
e_{\lambda,r}$ is analytic in $\lambda\in\mathcal D_\alpha$ and
invertible for all sufficiently large $r>0$. By analyticity in
$\lambda$ we conclude that $\mathsf e_{\lambda,r}(\zeta,\eta)$ is a
complex symmetric matrix, the Schwarz reflection principle gives
$\overline{\mathsf e_{\lambda,r}}={\mathsf
e_{\overline{\lambda},r}}$.

Differentiating the equality~\eqref{rrr}, we see that the matrices
$\partial_\zeta^p\partial_\eta^q \mathsf e_{\lambda,r} $ and
$\partial_\zeta^p\partial_\eta^q \mathsf e^{-1}_{\lambda,r}$ are
analytic in $\lambda\in\mathcal D_\alpha$. Moreover,
from~\eqref{lim},~\eqref{nstab}, and our assumptions on $\varkappa$
we obtain
$$
\left\|\partial^p_\zeta\partial^q_\eta\bigl(\mathsf
e^{-1}_{\lambda,r}(\zeta,\eta)- \diag\bigl\{ \bigl(1+\lambda
\bigr)^{-2},\operatorname{Id}\bigr\}\bigr)\right\|\to 0\text{ as
}\zeta\to+\infty,
$$
\begin{equation}\label{cont+}
{\begin{aligned}
&\left\|\partial^p_\zeta\partial^q_\eta\bigl(\mathsf
e^{-1}_{\lambda,r}(\zeta,\eta)- \diag\bigl\{ \bigl(1+\lambda \mathsf
s_r'(\zeta,\eta)\bigr)^{-2},\operatorname{Id}
\bigr\}\bigr)\right\|\leq c(r),
\\
& \text{ where } p+|q|\leq 1,  \text{ and } \ c(r)\to 0 \text{ as }
r\to+\infty.
\end{aligned}}
\end{equation}
Here $(\zeta,\eta)\in\overline{\mathcal C}$ and $\mathsf
s_r'(\zeta,\eta)$ stands for the function $s'(x-r)$ written in the
coordinates $(\zeta,\eta)$. We extend $\mathsf s_r'$ from
$\overline{\mathcal C}$ to $\overline{\mathcal G}$ by zero. Note
that the estimate~\eqref{cont+} remains valid for all
$(\zeta,\eta)\in\overline{\mathcal G}$ and the constant $c(r)$ is
independent of $\lambda\in\mathcal D_\alpha$ and
$(\zeta,\eta)\in\overline{\mathcal G}$. Now we see that for all
sufficiently large $r>0$  the differential operator~\eqref{v2} is
well-defined for all $\lambda$ in the disk $\mathcal D_\alpha$, and
its coefficients are subjected to the estimate~\eqref{cont+}.

\section{Analytic families of operators}\label{secAF}
In this section we study the unbounded operator $\Delta_{\lambda,r}$
in the Hilbert space $L^2(\mathcal G)$ with the usual norm
$$
\|u;
L^2(\mathcal G)\|=\left(\int_{\mathcal G}
|u(\zeta,\eta)|^2\,d\zeta\,d\eta\right)^{1/2}.
$$
  The operator
$\Delta_{\lambda,r}$ is initially defined on the set
$C_0^\infty(\overline{\mathcal G})$. We show that the operator is
closable, and the closure defines an analytic family $\mathcal
D_\alpha\ni\lambda\mapsto\Delta_{\lambda,r}$ of type $(\rm
B)$~\cite{Kato}. In a standard way this implies that the resolvent
$(\Delta_{\lambda,r}-\mu)^{-1}$ is an analytic function of $\lambda$
and $\mu$ on some open subset of $\mathcal D_\alpha\times\Bbb C$.
The latter fact will be used for justification of a limiting
absorption principle in Section~\ref{s6}.

We intend to show that the operator $\Delta_{\lambda,r}$ in
$L^2(\mathcal G)$ with the domain $C_0^\infty(\overline{\mathcal
G})$ is sectorial and to study its Friedrichs extension. With the
operator $\Delta_{\lambda,r}$ we
 associate
 the quadratic form
\begin{equation}\label{qfD}
q_{\lambda,r}[u,u]=\int_{\mathcal G} \bigl\langle\bigl({\det \mathsf
e_{\lambda,r} }\bigr)^{1/2}\mathsf e_{\lambda,r}^{-1} \nabla_{\zeta
\eta}u, \nabla_{\zeta \eta}\bigl({\det \mathsf
e_{\overline{\lambda},r} }\bigr)^{-1/2}
u\bigr\rangle\,d\zeta\,d\eta,
\end{equation}
where $\langle\cdot,\cdot\rangle$ is the Hermitian inner product in
$\mathbb C^{n+1}$ and $u\in C_0^\infty(\overline{\mathcal G})$. Let
$(\cdot,\cdot)$ stand for the inner product in $L^2(\mathcal G)$.
 We represent the
quadratic form as follows:
\begin{equation}\label{qf}
\begin{aligned}
q_{\lambda,r}[u& ,u]=(-\nabla_{\zeta \eta}\cdot\mathsf
e_{\lambda,r}^{-1} \nabla_{\zeta \eta}u,u)\\
&+\int_{\mathcal G}\bigl\langle\bigl({\det \mathsf e_{\lambda,r}
}\bigr)^{1/2}\mathsf e_{\lambda,r}^{-1} \nabla_{\zeta \eta}u,
u\nabla_{\zeta \eta}\bigl({\det \mathsf e_{\overline{\lambda},r}
}\bigr)^{-1/2}\bigr\rangle\,d\zeta\,d\eta.
\end{aligned}
\end{equation}
 For the first term in the right hand side
of~\eqref{qf} we prove the following lemma.
\begin{lemma}\label{l1}
Assume that $r>0$ is sufficiently large. Then  there exist
$\varphi<\pi/2$ and $\delta>0$ such that for all $\lambda\in\mathcal
D_\alpha$ and $u\in C_0^\infty(\overline{\mathcal G})$ we have
$$
|\arg (- \nabla_{\zeta \eta}\cdot\mathsf e_{\lambda,r}^{-1}
\nabla_{\zeta \eta}u,u) |\leq\varphi,\quad \delta(\Delta u,u)\leq
\Re(-\nabla_{\zeta \eta}\cdot\mathsf e_{\lambda,r}^{-1}
\nabla_{\zeta \eta}u,u)\leq \delta^{-1} (\Delta u,u).
$$
In other words, the form $(- \nabla_{\zeta \eta}\cdot\mathsf
e_{\lambda,r}^{-1} \nabla_{\zeta \eta}u,u)$ is sectorial.
\end{lemma}
\begin{proof} It is clear that
\begin{equation}\label{eq11+}
\begin{aligned}
&(- \nabla_{\zeta \eta}\cdot\mathsf e_{\lambda,r}^{-1} \nabla_{\zeta
\eta}u,u)=\int_{\mathcal G}\bigl\langle\mathsf e_{\lambda,r}^{-1}
\nabla_{\zeta \eta}u, \nabla_{\zeta
\eta}u\bigr\rangle\,d\zeta\,d\eta.
 \end{aligned}
\end{equation}
Let us estimate the numerical range of the matrix ${\mathsf
e}^{-1}_{\lambda,r}(\zeta,\eta)$. We shall rely on the
estimate~\eqref{cont+}. Let $\xi=\nabla_{\zeta
\eta}u(\zeta,\eta)\in\mathbb C^{n+1}$. Observe that by virtue of
$0\leq s_r'(\zeta,\eta)\leq 1$ and $|\lambda|<\sin\alpha<2^{-1/2}$
we have
 \begin{equation}\label{start}
 \begin{aligned}
&|\xi|^2/4\leq \Bigl|\overline {\xi}\cdot
\diag\bigl\{\bigl(1+\lambda \mathsf
s_r'(\zeta,\eta)\bigr)^{-2},\operatorname{Id}\bigr\}\xi\Bigr|\leq 12
|\xi|^2,
 \\
& \Bigl|\arg \Bigl(\overline {\xi}\cdot \diag\bigl\{\bigl(1+\lambda
\mathsf
s_r'(\zeta,\eta)\bigr)^{-2},\operatorname{Id}\bigr\}\xi\Bigr)\Bigr|<2\alpha.
\end{aligned}
 \end{equation}
Since $r$ is sufficiently large, the constant $c (r)$
in~\eqref{cont+} is small. In particular $c(r)$ meets the estimate
$4(n+1)^{2} c (r)\leq\sin(\sigma/2)$ with some
$\sigma\in(0,\pi/2-2\alpha)$. Then~\eqref{cont+} together
with~\eqref{start} gives
\begin{equation}\label{st}
 \Bigl|\arg \Bigl(\overline {\xi}\cdot \mathsf
e^{-1}_{\lambda,r}(\zeta,\eta){\xi}\Bigr)\Bigr|\leq\varphi<\pi/2,\quad
\delta |\xi|^2\leq\Bigl|\overline {\xi}\cdot
\mathsf e^{-1}_{\lambda,r}(\zeta,\eta){\xi}\Bigr|\leq \delta^{-1} |\xi|^2,\\
\end{equation}
uniformly in $\lambda\in\mathcal D_\alpha$ and
$(\zeta,\eta)\in\overline{\mathcal G}$, where $\varphi=
2\alpha+\sigma$ and
$$
 \delta=\min\Bigl\{1/4-(n+1)^2 c(r),\bigl(12+(n+1)^2
c(r)\bigr)^{-1} \Bigr\}.
$$
Taking into account~\eqref{eq11+} we complete the proof.
\end{proof}

\begin{remark}\label{rem r} Throughout the paper we say that $r>0$ is
sufficiently  large if  the matrix $\mathsf
e_{\lambda,r}(\zeta,\eta)$ is invertible and its inverse meets the
estimates~\eqref{st} uniformly in
$(\zeta,\eta)\in\overline{\mathcal G}$ and $\lambda\in\mathcal
D_\alpha$.
\end{remark}

In the next lemma we show that in the right hand side of~\eqref{qf}
the second term has an arbitrarily small relative bound  with
respect to the first term uniformly in $\lambda\in\mathcal
D_\alpha$.
\begin{lemma}\label{r bound} For any $\epsilon>0$ and $u\in C_0^\infty(\overline{\mathcal G})$  the estimate
$$
\begin{aligned}
 \Bigl|\int_{\mathcal G}\bigl\langle\bigl({\det \mathsf e_{\lambda,r}
}\bigr)^{1/2}\mathsf e_{\lambda,r}^{-1} \nabla_{\zeta \eta}u, &
u\nabla_{\zeta \eta}\bigl({\det \mathsf e_{\overline{\lambda},r}
}\bigr)^{-1/2}\bigr\rangle\,d\zeta\,d\eta\Bigr|
\\
& \leq \epsilon |( -\nabla_{\zeta \eta}\cdot\mathsf
e_{\lambda,r}^{-1} \nabla_{\zeta \eta}u,u)
 |+C\epsilon^{-1}\|u; L^2(\mathcal G)\|^2
\end{aligned}
$$
holds, where the constant $C$ is independent of $\epsilon$, $u$, and
$\lambda\in\mathcal D_\alpha$.
\end{lemma}
\begin{proof} We have
\begin{equation}\label{add++}
\begin{aligned}
\Bigl | &\int_{\mathcal G}\bigl\langle\bigl({\det \mathsf
e_{\lambda,r} }\bigr)^{1/2}\mathsf e_{\lambda,r}^{-1} \nabla_{\zeta
\eta}u, u\nabla_{\zeta \eta}\bigl({\det \mathsf
e_{\overline{\lambda},r}
}\bigr)^{-1/2}\bigr\rangle\,d\zeta\,d\eta\Bigr|
\\
& \leq \texttt{C} \Bigl(\int_{\mathcal G}|\bigl\langle\mathsf
e_{\lambda,r}^{-1} \nabla_{\zeta \eta}u, \nabla_{\zeta
\eta}\bigl({\det \mathsf e_{\overline{\lambda},r}
}\bigr)^{-1/2}\bigr\rangle|^2\,d\zeta\,d\eta \Bigr)^{1/2}
\Bigl(\int_{\mathcal G} |u|^2\,d\zeta\,d\eta \Bigr)^{1/2}
\\
&\leq \texttt{C}\Bigl(\tilde \epsilon\int_{\mathcal
G}|\bigl\langle\mathsf e_{\lambda,r}^{-1} \nabla_{\zeta \eta}u,
\nabla_{\zeta \eta}\bigl({\det \mathsf e_{\overline{\lambda},r}
}\bigr)^{-1/2}\bigr\rangle|^2\,d\zeta\,d\eta+
\tilde\epsilon^{-1}\|u; L^2(\mathcal G)\|^2\Bigr)
\end{aligned}
\end{equation}
with  arbitrarily small $\tilde\epsilon>0$ and
$\texttt{C}=\sup_{(\zeta,\eta)}\bigl({\det \mathsf e_{\lambda,r}
(\zeta,\eta)}\bigr)^{1/2}<\infty$, cf.~\eqref{cont+}.

From \eqref{st} it follows that
$$
\bigl|\Im\{\overline {\xi}\cdot \mathsf
e^{-1}_{\lambda,r}(\zeta,\eta){\xi}\bigr\}\bigr|\leq
(\tan\varphi)\Re \{\overline {\xi}\cdot \mathsf
e^{-1}_{\lambda,r}(\zeta,\eta){\xi}\},
$$
where the form $\Re \{\overline {\xi}\cdot \mathsf
e^{-1}_{\lambda,r}(\zeta,\eta){\xi}\bigr\}$ defines an inner product
in $\Bbb C^{n+1}$. This and the Cauchy-Schwarz inequality give
\begin{equation}\label{add+}
\bigl|\overline {\tau}\cdot \mathsf
e^{-1}_{\lambda,r}(\zeta,\eta){\xi}\bigr|^2 \leq
(1+\tan\varphi)^{2}\Re \{\overline {\xi}\cdot \mathsf
e^{-1}_{\lambda,r}(\zeta,\eta){\xi}\bigr\}\,\Re \{\overline
{\tau}\cdot \mathsf e^{-1}_{\lambda,r}(\zeta,\eta){\tau}\bigr\},
\end{equation}
cf.~\cite[Chapter VI.2]{Kato}. We  substitute $\xi=\nabla_{\zeta
\eta}u(\zeta,\eta)$ and $\tau=\nabla_{\zeta \eta}\bigl({\det \mathsf
e_{{\lambda},r} (\zeta,\eta)}\bigr)^{-1/2}$. Thanks to~\eqref{cont+}
we have the uniform bound
$$
|\nabla_{\zeta \eta}\bigl({\det \mathsf e_{{\lambda},r}
(\zeta,\eta)}\bigr)^{-1/2}|^2\leq c,\quad \lambda\in\mathcal
D_\alpha,\ (\zeta,\eta)\in\overline{\mathcal G}.
$$
This together with~\eqref{add+} and~\eqref{st} implies
$$
|\bigl\langle\mathsf e_{\lambda,r}^{-1} \nabla_{\zeta \eta}u,
\nabla_{\zeta \eta}\bigl({\det \mathsf e_{\overline{\lambda},r}
}\bigr)^{-1/2}\bigr\rangle|^2\leq
c\delta^{-1}(1+\tan\varphi)^{2}\Re\bigl\langle\mathsf
e_{\lambda,r}^{-1} \nabla_{\zeta \eta}u, \nabla_{\zeta
\eta}u\bigr\rangle.
$$
Now we make use of~\eqref{add++} and establish the assertion for
$C=\texttt{C}^2 c\delta^{-1}(1+\tan\varphi)^2$ and an arbitrarily
small $\epsilon=\texttt{C} c\delta^{-1} (1+\tan\varphi)^2\tilde
\epsilon$.
 \end{proof}

As a consequence of the equality~\eqref{qf} and
Lemmas~\ref{l1},~\ref{r bound} for all sufficiently large $r>0$ we
obtain
\begin{equation}\label{i1}
 |\arg\bigl( q_{\lambda,r} [u,u]+\gamma \|u; L^2(\mathcal G)\|^2\bigr)|\leq \varphi<\pi/2,
\end{equation}
\begin{equation}\label{i2}
\begin{aligned}
\delta q_{0,r}[u,u]-\gamma\|u; L^2(\mathcal G)\|^2 & \leq \Re
q_{\lambda,r} [u,u]
 \leq \delta^{-1}(q_{0,r}[u,u]+\|u; L^2(\mathcal G)\|^2)
\end{aligned}
\end{equation}
with some angle $\varphi$ and some positive constants $\delta$ and
$\gamma$, which are independent of $u\in
C_0^\infty(\overline{\mathcal G})$ and $\lambda\in\mathcal
D_\alpha$. The symmetric form $q_{0,r}[u,u]=\int_{\mathcal G}\langle
\nabla_{\zeta \eta}u, \nabla_{\zeta
\eta}u\bigr\rangle\,d\zeta\,d\eta $ is independent of $r$, it
corresponds to the Dirichlet Laplacian $\Delta\equiv\Delta_{0,r}$.
Clearly, $q_{\lambda,r} [u,u]=(\Delta_{\lambda,r}u,u)$.
Estimate~\eqref{i1} implies that the numerical range
$$\{\mu\in\Bbb C: \mu=(\Delta_{\lambda,r}u,u), u\in
C_0^\infty(\overline{\mathcal G}), \|u;L^2(\mathcal G)\|=1\}$$ is a
subset of the sector $\{\mu\in\Bbb
C:|\arg(\mu+\gamma)|\leq\varphi<\pi/2\}$. By definition this means
that the operator $\Delta_{\lambda,r}$ with the domain
$C_0^\infty(\overline{\mathcal G})$ is sectorial.

We introduce the Sobolev space
$\lefteqn{\stackrel{\circ}{\phantom{\,\,>}}}H^1(\mathcal G)$ as the
completion of the set $C_0^\infty(\overline{\mathcal G})$ with
respect to the norm
$$
\|u; \lefteqn{\stackrel{\circ}{\phantom{\,\,>}}}H^1(\mathcal
G)\|=\sqrt{q_{0,r}[u,u]+\|u;L^2(\mathcal G)\|^2}.
$$
Recall that {\it i}) a sequence $\{u_j\}$ is said to be
$q_{\lambda,r}-$con\-ver\-gent, if $u_j$ is in the domain of $
q_{\lambda,r}$, $\|u_j-u; L^2(\mathcal G)\|\to 0$ and $q_{\lambda,r}
[u_j-u_m,u_j-u_m]\to 0$ as $j,m\to \infty$; {\it ii}) the  form $
q_{\lambda,r}$ is closed, if every $ q_{\lambda,r}$-convergent
sequence $\{u_j\}$ has a limit $u$ in the domain of $
q_{\lambda,r}$, and $q_{\lambda,r}[u-u_j,u-u_j]\to 0$.
From~\eqref{i1},~\eqref{i2} it immediately follows that
$q_{\lambda,r}$ with the domain
$\lefteqn{\stackrel{\circ}{\phantom{\,\,>}}}H^1(\mathcal G)$ is a
 closed densely defined sectorial form~\cite{Kato,Simon Reed iv}, and its sector $\{\mu\in\Bbb C:|\arg(\mu+\gamma)|\leq\varphi\}$ is independent of
$\lambda\in\mathcal D_\alpha$. As known~\cite[Chapter VI.2.1]{Kato},
to every closed densely defined sectorial form  there corresponds a
unique m-sectorial operator. Namely, to the form $q_{\lambda,r}$
there corresponds a unique m-sectorial operator $\Delta_{\lambda,r}$
in $L^2(\mathcal M)$ such that its sector  is the sector of
$q_{\lambda,r}$, the domain $\mathrm D (\Delta_{\lambda,r})$ is
dense in $\lefteqn{\stackrel{\circ}{\phantom{\,\,>}}}H^1(\mathcal
G)$, and $\mathsf q_{\lambda,r}[u,v]=(\Delta_{\lambda,r}u,v)$ for
all $u\in \mathrm D(\Delta_{\lambda,r})$ and
$v\in\lefteqn{\stackrel{\circ}{\phantom{\,\,>}}}H^1(\mathcal G)$.
 (Here and elsewhere
m-sectorial means that the numerical range $\{ \mu=(Au,u)_{\mathcal
H}\in\Bbb C: u\in\mathrm D(A),(u,u)_{\mathcal H}=1 \}$ and the
spectrum $\sigma(A)$ of a closed unbounded operator $A$ in a Hilbert
space $\mathcal H$ with the inner product $(\cdot,\cdot)_{\mathcal
H}$ both lie in some sector $\{\mu\in \Bbb C: \arg(\mu-\gamma)\leq
\varphi\}$ with $\gamma\in\Bbb R$ and $\varphi<\pi/2$.) In
particular, to the symmetric nonnegative form $ q_{0,r}$ there
corresponds the selfadjoint Dirichlet Laplacian
$\Delta\equiv\Delta_{0,r}$. The m-sectorial operator
$\Delta_{\lambda,r}$ in $L^2(\mathcal M)$ is the Friedrichs
extension of the sectorial operator $\Delta_{\lambda,r}$ defined on
 $C_0^\infty(\overline{\mathcal G})$, see~\cite[Chapter
VI.2.3]{Kato}. As we show in assertion 1 of the next proposition the
set $C_0^\infty(\overline{\mathcal G})$ is a core of the Friedrichs
extension.
 \begin{proposition}\label{sectorial} Assume that $r>0$ is sufficiently large. Then the following
 assertions are valid.
\begin{enumerate}
\item For  $\lambda\in\mathcal
 D_\alpha$ the m-sectorial operator $\Delta_{\lambda, r}$ in $L^2(\mathcal
G)$ with the domain $\mathrm D(\Delta_{\lambda, r})$ is the closure
of the operator $\Delta_{\lambda,r}$ defined on the set
$C_0^\infty(\overline{\mathcal G})$.

\item The family of m-sectorial operators $\mathcal D_\alpha\ni\lambda\mapsto
\Delta_{\lambda,r}$ in $L^2(\mathcal G)$ is an analytic family of
type $(\rm B)$.

\item  The resolvent $\Gamma\ni (\lambda,\mu)\mapsto (\Delta_{\lambda,r}-\mu)^{-1}: L^2(\mathcal G)\to L^2(\mathcal G)$ is an analytic
    function of two variables on the set $\Gamma=\bigl\{(\lambda,\mu): \lambda\in\mathcal D_\alpha, \mu\in\Bbb C\setminus\sigma(
    \Delta_{\lambda,r})\bigr\}$, where $\sigma(
    \Delta_{\lambda,r})$ is the spectrum of $\Delta_{\lambda,r}$.
\end{enumerate}
\end{proposition}
\begin{proof} Consider the domain $\mathrm D
(\Delta_{\lambda,r})$ as a Hilbert space endowed with the graph norm
$\|u;\mathrm D (\Delta_{\lambda,r})\|=\|u; L^2(\mathcal
G)\|+\|\Delta_{\lambda,r}u;L^2(\mathcal G)\|$. Let $\mu$ be a point
outside of the sector of the m-sectorial operator
$\Delta_{\lambda,r}$. Then the set
$$
\mathrm C(\Delta_{\lambda,r})=\{u: u=(\Delta_{\lambda,r}-\mu)^{-1}
f, f\in C_0^\infty(\overline{\mathcal G})\}
$$
is dense in $\mathrm D (\Delta_{\lambda,r})$  because the resolvent
$(\Delta_{\lambda,r}-\mu)^{-1}: L^2(\mathcal G)\to \mathrm
D(\Delta_{\lambda,r})$ is bounded and the set
$C_0^\infty(\overline{\mathcal G})$ is dense in $L^2(\mathcal G)$.
 From~\eqref{st} it follows the estimate
 \begin{equation}\label{strong ell}
 \Re(\xi\cdot\mathsf e^{-1}_{\lambda,r}(\zeta,\eta)\xi)\geq
 c|\xi|^2,\quad\xi\in\Bbb R^{n+1}, \lambda\in\mathcal D_\alpha,(\zeta,\eta)\in \overline{\mathcal G},
 \end{equation}
 on the principal symbol of  $\Delta_{\lambda,r}$, where $c>0$. Hence $\Delta_{\lambda,r}$ is a strongly
 elliptic operator. As is well-known, a strongly elliptic operator and the Dirichlet
boundary condition set up an elliptic boundary value problem,
e.g.~\cite{Lions Magenes}. The usual argument on the regularity of
solutions to the elliptic boundary value problems~\cite{Lions
Magenes,KozlovMazyaRossmann} implies that the set $\mathrm
C(\Delta_{\lambda,r})$ consists of smooth in $\overline{\mathcal G}$
functions $u$ with $u\upharpoonright_{\partial\mathcal G}=0$.
Multiplying $u\in \mathrm C(\Delta_{\lambda,r})$ by appropriate
cutoff functions $\chi_j$ with expanding compact supports $\supp
\chi_j\subset \supp\chi_{j+1}$, it is easy to see that for any $u\in
\mathrm C(\Delta_{\lambda,r})$ there is a sequence $\{\chi_j
u\}_{j=1}^\infty$ such that $\chi_j u\in
C_0^\infty(\overline{\mathcal G})$ tends to $u$ in $\mathrm
D(\Delta_{\lambda,r})$ as $j\to +\infty$.  Assertion 1 is proven.

 The  family $\mathcal D_\alpha\ni\lambda\mapsto q_{\lambda,r}$ is
analytic in the sense of Kato (i.e. $q_{\lambda,r}$ is a closed
densely defined sectorial form, its domain
$\lefteqn{\stackrel{\circ}{\phantom{\,\,>}}}H^1(\mathcal G)$ is
independent of  $\lambda$, and the function $\mathcal
D_\alpha\ni\lambda\mapsto q_{\lambda,r}[u,u]$ is analytic for any
$u\in\lefteqn{\stackrel{\circ}{\phantom{\,\,>}}}H^1(\mathcal G)$).
By definition~\cite{Kato, Simon Reed iv} this means that the family
of m-sectorial operators $\mathcal D_\alpha\ni\lambda\mapsto
\Delta_{\lambda,r}$ is an analytic family of type $(\rm B)$. This
proves assertion 2.

As is well-known~\cite{Kato, Simon Reed iv},  any analytic family of
type $(\rm B)$ is also an analytic family of operators in the sense
of Kato. Now a standard argument justifies assertion~3; see, e.g.,
\cite[Theorem XII.7]{Simon Reed iv}.
\end{proof}

\section{Spaces of analytic functions}\label{s3} We have shown that
for all $\mu$ outside the sector of the family $\mathcal
D_\alpha\ni\lambda\mapsto\Delta_{\lambda,r}$ of m-sectorial
operators the resolvent $(\Delta_{\lambda,r}-\mu)^{-1}$ is an
analytic function of $\lambda$, see Proposition~\ref{sectorial}.3.
In order to get some relations between the resolvents
$(\Delta-\mu)^{-1}$ and $(\Delta_{\lambda,r}-\mu)^{-1}$ we will use
a sufficiently large Hilbert space $\mathcal H_\alpha(\mathcal G)$
of analytic functions
\begin{equation}
\label{stars}\mathcal D_\alpha\ni\lambda\mapsto
f\circ\vartheta_{\lambda,r}\in L^2(\mathcal G);
\end{equation}
 here $\vartheta_{\lambda,r}$ is the complex scaling~\eqref{sc}.  The goal of this section is two-fold:
\begin{enumerate}
\item To
introduce a Hilbert space $\mathcal H_\alpha(\mathcal G)$, which is
sufficiently large in the sense that for any $\lambda\in\mathcal
D_\alpha$ and $r>0$ the set $\{ f\circ \vartheta_{\lambda,r}\in
L^2(\mathcal G):f\in\mathcal H_\alpha(\mathcal G)\}$ is dense in
$L^2(\mathcal G)$;
\item To derive the uniform in $\lambda\in\mathcal
D_\alpha$ and $f\in\mathcal H_\alpha(\mathcal G)$ estimate
\begin{equation}\label{stars+}
\|f\circ\vartheta_{\lambda,r}; L^2(\mathcal G)\|\leq C_r\|f;
\mathcal H_\alpha(\mathcal G) \|,\quad r>0.
\end{equation}
\end{enumerate}

Introduce the Hardy class $\mathfrak H(\Bbb S_\alpha)$  of all
analytic functions $\Bbb S_\alpha\ni z\mapsto F(z)\in L^2(\Omega)$
 satisfying the uniform in $\psi$ estimate
\begin{equation}\label{hardy}
\int_0^\infty \| F(e^{i\psi} x);L^2(\Omega)\|^2\,dx\leq
C_F,\quad\psi\in(-\alpha,\alpha).
\end{equation}
 Below we cite some facts
from the theory of Hardy classes.
\begin{lemma}\label{l3}
\begin{enumerate}
\item Every $ F\in \mathfrak H(\Bbb S_\alpha)$ has boundary limits $
F_\pm\in L^2(\Bbb R_+\times{\Omega})$ such that
$$
\int_0^\infty\| F(e^{i\psi}x)- F_\pm(x);L^2(\Omega)\|^2\,dx\to 0
\text{  as  }\psi\to\pm\alpha.
$$

\item The Hardy class $\mathfrak H(\Bbb S_\alpha)$ endowed with the
norm
$$
\| F; \mathfrak H(\Bbb S_\alpha)\|=\bigl(\|F_-;L^2(\Bbb
R_+\times\Omega)\|^2+\|F_+;L^2(\Bbb
R_+\times{\Omega})\|^2\bigr)^{1/2}
$$
is a Hilbert space.

\item For every compact set $\mathfrak K\subset\Bbb S_\alpha$ there
is an independent of $F\in \mathfrak H(\Bbb S_\alpha)$ constant
$C(\mathfrak K)$ such that for all $z\in\mathfrak K$ we have
 $$
 \| F(z); L^2(\Omega)\|\leq C(\mathfrak K)\|F; \mathfrak H(\Bbb S_\alpha)\|.
 $$

\item For any $F\in \mathfrak H(\Bbb S_\alpha)$, $z\in {\Bbb
S_\alpha}$, and $\psi\in(-\alpha,\alpha)$ we have
$$
\int_0^\infty \| F(z+ e^{i\psi} x);L^2(\Omega)\|^2\,dx\leq \| F;
\mathfrak H(\Bbb S_\alpha)\|^2.
$$
\end{enumerate}
\end{lemma}
\begin{proof}
Assertions 1--3 are direct consequences of well-known facts from the
theory of Hardy spaces of functions analytic in strips,
e.g.~\cite{Ros Rov}. In fact, the proof reduces to the conformal
mapping of the sector~$\Bbb S_\alpha$ to the strip $\{z\in\Bbb
C:-\alpha<\Im z< \alpha\}$, we omit the details.

Let us prove  assertion~4. A standard argument, see
e.g.~\cite{Titchmarsh}, shows that any $ F\in\mathfrak H(\Bbb
S_\alpha)$ can be represented  by the Cauchy integral
\begin{equation}\label{cauchy}
F(z)=\frac 1 {2\pi i}\int\limits_{-\infty}^{+\infty}
\frac{e^{i\alpha} F_+(x)}{z-e^{i\alpha}x}\,dx-\frac 1 {2\pi
i}\int\limits_{-\infty}^{+\infty} \frac{e^{-i\alpha} F_-(x) }
   {z-e^{-i\alpha}x}\,dx,\quad z\in\Bbb S_\alpha,
\end{equation}
where we assume that the boundary limits $F_\pm(x)$ are extended  to
$x<0$ by zero. The integrals are absolutely convergent in the space
$L^2(\Omega)$. As is well-known~\cite{Ros Rov}, the first integral
in~\eqref{cauchy} defines an element $f_+$ of the Hardy space
$\mathfrak H (\Bbb C^-_\alpha)$ of $L^2(\Omega)$-valued functions
analytic in the half-plane $\Bbb C^-_\alpha=\{z\in\Bbb C: \Im(
e^{-i\alpha}z) <0\}$. As shown in~\cite{Van Winter I,Van Winter},
the norm in $\mathfrak H(\Bbb C^-_\alpha)$ can be defined by the
equality
$$
\|f_+;\mathfrak H(\Bbb
C^-_\alpha)\|=\left(\sup_{\psi\in(\alpha-\pi,\alpha)}\int_0^\infty
\|f_+(z_++e^{i\psi}x); L^2(\Omega)\|^2\,dx\right)^{1/2},\quad
e^{-i\alpha}z_+\in\Bbb R.
$$
Then $\|f_+;\mathfrak H(\Bbb C^-_\alpha)\|= \|F_+; L^2(\Bbb
R_+\times{\Omega})\|$. Similarly, the second integral
in~\eqref{cauchy} defines a function $f_-$ from the Hardy space
$\mathfrak H(\Bbb C^+_{\alpha})$ in $\Bbb C^+_\alpha=\{z\in\Bbb C:
\Im (e^{i\alpha}z)>0\}$, and $\|f_-;\mathfrak H(\Bbb C^+_\alpha)\|=
\|F_-; L^2(\Bbb R_+\times{\Omega})\|$, where
$$\|f_-;\mathfrak H(\Bbb
C^+_\alpha)\|=\left(\sup_{\psi\in(-\alpha,\pi-\alpha)}\int_0^\infty
\|f_-(z_-+e^{i\psi}x); L^2(\Omega)\|^2\,dx\right)^{1/2},\quad
e^{i\alpha}z_-\in\Bbb R.
$$
As a consequence,  for all $z\in{\Bbb S_\alpha}$ and
$\psi\in(-\alpha, \alpha)$ we have
$$
\begin{aligned}
\int_0^\infty \|F(z+ e^{i\psi} x);L^2(\Omega)\|^2\,dx\leq
\|f_+;\mathfrak H(\Bbb C^-_\alpha)\|^2+\|f_-;\mathfrak H(\Bbb
C^+_\alpha)\|^2\\=\|F_+; L^2(\Bbb
R_+\times{\Omega})\|^2+\|F_-;L^2(\Bbb R_+\times{\Omega})\|^2= \| F;
\mathfrak H(\Bbb S_\alpha)\|^2.
\end{aligned}
$$
\end{proof}

 Consider the algebra  $\mathscr E$ of all entire functions $\mathbb C\ni z\mapsto F(z)\in C_0^\infty({\Omega})$ with the following property:
in any sector $|\Im z|\leq (1-\epsilon) \Re z$ with $\epsilon>0$ the
value $\|F(z);L^2(\Omega)\|$ decays faster than any inverse power of
$\Re z$  as $\Re z\to+\infty$. Examples of functions $F\in \mathscr
E$ are $F(z)=e^{-\gamma z^2}P(z)$, where $\gamma>0$ and $P(z)$ is an
arbitrary polynomial  in $z$ with coefficients in
$C^\infty_0({\Omega})$. Clearly, $\mathscr E\subset\mathfrak H(\Bbb
S_\alpha)$. The next lemma is an adaptation of~\cite[Theorem
3]{Hunziker}, we omit the proof.

\begin{lemma}\label{p1} The set of functions
$\{\Bbb R_+\times{\Omega}\ni(x,y)\mapsto
F\circ\kappa_{\lambda,r}(x,y): F\in\mathscr E\}$
 is dense in the space $L^2(\Bbb R_+\times{\Omega})$ for any
$\lambda\in\mathcal D_\alpha$ and $r>0$. Here
$F\circ\kappa_{\lambda,r}(x,\cdot)$ stands for the value of the
entire function $z\mapsto F(z)$ at the point $z=x+\lambda s(x-r)$.
\end{lemma}

Now we are in position to prove the following proposition.
 \begin{proposition}\label{HS}\begin{enumerate}
\item The estimate
\begin{equation}\label{US}
\int_0^\infty\|F\circ\kappa_{\lambda,r}; L^2(\Omega)\|^2\,dx\leq
C_r\|F;\mathfrak H(\Bbb S_\alpha)\|^2,\quad r>0,
\end{equation}
 holds uniformly in $\lambda\in\mathcal D_\alpha$ and $
F\in\mathfrak H(\Bbb S_\alpha)$.

\item The space $\mathfrak H(\Bbb S_\alpha)$ is sufficiently large in the sense that for any
$\lambda\in\mathcal D_\alpha$ and $r>0$ the set $\{\Bbb
R_+\times{\Omega}\ni(x,y)\mapsto F\circ
\kappa_{\lambda,r}(x,y):F\in\mathfrak H(\Bbb S_\alpha)\}$ is dense
in $L^2(\Bbb R_+\times{\Omega})$.
\item For any $F\in \mathfrak H(\Bbb S_\alpha)$ and $r>0$ the function $\mathcal D_\alpha\in\lambda\mapsto F\circ\kappa_{\lambda,r}\in L^2(\Bbb R_+\times{\Omega})$ is
analytic.
\end{enumerate}
\end{proposition}
The proof is preceded by a discussion. By Proposition~\ref{HS}.1 for
any $r>0$ the complex scaling $\kappa_{\lambda,r}$ induces the
uniformly bounded injection
\begin{equation}\label{inj}
\mathfrak H(\Bbb S_\alpha)\ni F\mapsto F\circ\kappa_{\lambda,r} \in
L^2(\Bbb R_+\times{\Omega}),\quad \lambda\in\mathcal D_\alpha.
\end{equation}
Any function  $F\in\mathfrak H(\Bbb S_\alpha)$ can be reconstructed
from its trace $x\mapsto F\circ\kappa_{\lambda,r}(x,\cdot)\in
L^2(\Omega)$ by analytic continuation from the curve $\{x+\lambda
s(x-r)\in\Bbb C: x>0\}$ to the sector $\Bbb S_\alpha$. Therefore we
can always identify the space $\mathfrak H(\Bbb S_\alpha)$ with the
 range  of injection~\eqref{inj}. By Proposition~\ref{HS}.2 the range is dense
 in $L^2(\Bbb R_+\times{\Omega})$.

\begin{proof} 1. For all $\lambda\in\mathcal D_\alpha$ the curve $\{x+\lambda s(x-r)\in\Bbb C: x>0 \}$ lies in the sector
$\Bbb S_\alpha$.  Observe that this curve is differ from a ray only
inside an independent of $\lambda\in\mathcal D_\alpha$ compact
subset $\mathfrak K_r\subset\Bbb S_\alpha$. Now the uniform in
$\lambda\in\mathcal D_\alpha$ and $F\in\mathfrak H(\Bbb S_\alpha)$
estimate~\eqref{US} follows from assertions 3 and 4 of
Lemma~\ref{l3}.

2.  The assertion is an immediate consequence of the embedding
$\mathscr E\subset\mathfrak H(\Bbb S_\alpha)$, Lemma~\ref{p1}, and
the estimate~\eqref{US}.

3. It is easy to  see  that the function  $\mathcal
D_\alpha\ni\lambda\mapsto F\circ\kappa_{\lambda,r}\in L^2(\Bbb
R_+\times{\Omega})$ is weakly (and therefore strongly) analytic.
\end{proof}

Introduce the Hilbert space $\mathcal H_\alpha(\mathcal G)$ with the
norm
$$
\|f; \mathcal H_\alpha(\mathcal G) \|=\|f; L^2(\mathcal
G)\|+\|f\circ\varkappa^{-1}; \mathfrak H(\Bbb S_\alpha) \|.
$$
The space $\mathcal H_\alpha(\mathcal G)$ consists of all functions
$f\in L^2(\mathcal G)$ such that $f\circ\varkappa^{-1}$ is an
element of the Hardy space $\mathfrak H(\Bbb S_\alpha)$; here
$\varkappa$ is the diffeomorphism~\eqref{diff}. As a consequence of
Proposition~\ref{HS} and  definition~\eqref{sc} of the complex
scaling $\vartheta_{\lambda,r}$  we immediately get the following
assertions.
\begin{corollary}\label{dense}
\begin{enumerate}
\item For any $r>0$ the estimate~\eqref{stars+}
holds with an independent of $f\in\mathcal H_\alpha(\mathcal G)$ and
$\lambda\in\mathcal D_\alpha$ constant $C_r$.
\item For any
$\lambda\in\mathcal D_\alpha$ and $r>0$ the set $\{ f\circ
\vartheta_{\lambda,r}\in L^2(\mathcal G):f\in\mathscr
H_\alpha(\mathcal G)\}$ is dense in the space $L^2(\mathcal G)$.
\item For any $f\in \mathcal H_\alpha(\mathcal G)$ and $r>0$ the function~\eqref{stars} is
analytic.
\end{enumerate}
\end{corollary}

\section{Limiting absorption principle}\label{s6}
Introduce the Sobolev space $H_0^2(\mathcal G)$ of functions
satisfying the homogeneous Dirichlet boundary condition on
$\partial\mathcal G$ as the completion of the core
$C_0^\infty(\overline{\mathcal G})$ with respect to the graph norm
\begin{equation}\label{norm}
\|u;H^2_0(\mathcal G)\|=\|u;L^2(\mathcal G)\|+\|\Delta u;
L^2(\mathcal G)\|.
\end{equation}
(Integrating by parts and using the Cauchy-Schwarz inequality one
can easily see that the norm~\eqref{norm} is equivalent to the
traditional norm $(\sum_{\ell+|m|\leq
2}\|\partial_\zeta^\ell\partial_\eta^mu;L^2(\mathcal
G)\|^2)^{1/2}$.) By Proposition~\ref{sectorial}.1 the space
$H^2_0(\mathcal G)$ is the domain $\mathrm D(\Delta)$ of the
selfadjoint Dirichlet Laplacian $\Delta$.
 For the points $\mu\in \sigma(\Delta)$  the
resolvent $(\Delta-\mu-i\epsilon)^{-1}$ does not have limits in the
space of bounded operators  $\mathscr B( L^2(\mathcal G),
H^2_0(\mathcal G))$ as $\epsilon$ tends to zero from below
($\epsilon\uparrow 0$) or from above ($\epsilon\downarrow 0$).
However the limits may exist in the space of bounded operators
acting from a smaller source space  to a larger target space. As a
source space we take the Hilbert space $\mathcal H_\alpha(\mathcal
G)$ constructed in the previous section.
 As
a target space we take the reflexive Fr\'{e}chet space
$H^2_{0,\operatorname{loc}}(\mathcal G)$. The space
$H^2_{0,\operatorname{loc}}(\mathcal G)$
 consists of all distributions $u$ such that $\varrho u\in H^2_0(\mathcal G)$
 with any $\varrho\in C_c^\infty(\overline{\mathcal G})$,
 the topology in $H^2_{0,\operatorname{loc}}(\mathcal G)$ is induced
 by the family of seminorms $u\mapsto \|\varrho u;H^2_0(\mathcal
 G)\|$; here  $C_c^\infty(\overline{\mathcal G})$ is the set  of all smooth
 functions with compact supports in $\overline{\mathcal G}$.
\begin{theorem}\label{LA} Let $\sigma(\Delta_\Omega)$ stand for  the spectrum of the
selfadjoint Dirichlet Laplacian $\Delta_\Omega$ in $L^2(\Omega)$.
Assume that $\mu_0\in\Bbb R\setminus\sigma(\Delta_\Omega)$ is not an
eigenvalue of the selfadjoint Dirichlet Laplacian $\Delta$ in
$L^2(\mathcal G)$. Then the following assertions hold.
\begin{enumerate}
\item For all sufficiently large $r>0$ and $\lambda\in\mathcal D_\alpha\setminus\Bbb R$ the resolvent
$$
(\Delta_{\lambda,r}-\mu_0)^{-1}: L^2(\mathcal G)\to H^2_0(\mathcal
G)
$$
is a bounded operator.
\item The resolvent
$(\Delta-\mu_0-i\epsilon)^{-1}$, $\epsilon\gtrless 0$,
 viewed as a  bounded operator acting from $\mathcal H_\alpha(\mathcal G)$ to $H^2_{0,\operatorname{loc}}(\mathcal
G)$, has limits as $\epsilon\downarrow  0$ and $\epsilon \uparrow
0$.

\item Suppose that $r>0$ is sufficiently large, $\lambda\in \mathcal D_\alpha\setminus\Bbb R$, and $f\in\mathcal H_\alpha(\mathcal G)$.
Let $u_{\lambda,r}\in H^2_0(\mathcal G)$ be given by the equality
$u_{\lambda,r}=(\Delta_{\lambda,r}-\mu_0)^{-1}(f\circ\vartheta_{\lambda,r})$.
 Then the
outgoing  $u_-\in H^2_{0,\operatorname{loc}}(\mathcal G)$ and the
incoming $u_+\in H^2_{0,\operatorname{loc}}(\mathcal G)$ solutions
defined by the limiting absorption principle
\begin{equation}
u_+  =\lim_{\epsilon\uparrow 0}(\Delta-\mu_0-i\epsilon)^{-1} f,\quad
u_-=\lim_{\epsilon\downarrow 0}(\Delta-\mu_0-i\epsilon)^{-1} f,
\label{in out}
\end{equation}
meet the relation
 $$ u_{\lambda,r}\!\upharpoonright_{\mathcal G_r}=\left\{
                                                    \begin{array}{ll}
                                                      u_+\!\upharpoonright_{\mathcal G_r}, &  \Im\lambda<0; \\
                                                      u_-\!\upharpoonright_{\mathcal G_r}, &  \Im\lambda>0.
                                                    \end{array}
                                                  \right. $$
 Here the bounded domain $\mathcal
G_r$ is the same as in~\eqref{truncated}.
\end{enumerate}
\end{theorem}

The proof is preceded by a discussion. From Theorem~\ref{LA}.3 we
see that the equation $(\Delta_{\lambda,r}-\mu_0)u_{\lambda,
r}=f\circ\vartheta_{\lambda,r}$ with non-real parameter
$\lambda\in\mathcal D_\alpha$ describes infinite PMLs on $\mathcal
G\setminus\mathcal G_r$ for all $\mu_0$ satisfying the assumptions
of the theorem. The layers are perfectly matched in the sense that
for $\Im\lambda>0$ (resp. for $\Im\lambda<0$) $u_{\lambda,r}$
coincides in $\mathcal G_r$ with the outgoing solution $u_-$ (resp.
with  the incoming solution $u_+$). The PMLs are absorbing because
in contrast to $u_\pm$  the function $u_{\lambda, r}$ decays at
infinity in the mean as an element of $H^2_0(\mathcal G)$. In the
next section we will refine results of Theorem~\ref{LA} by showing
that under an additional assumption on $f\circ\vartheta_{\lambda,r}$
the solution $u_{\lambda,r}$ is of some exponential decay at
infinity. For instance, this assumption is a priori met for $f\in
L^2(\mathcal G)$  supported in $\mathcal G_r$. Then
$f\circ\vartheta_{\lambda,r}\equiv f$ and the operator
$\Delta_{\lambda,r}$ completely describes  infinite PMLs on
$\mathcal G\setminus\mathcal G_r$.
\begin{proof} The proof consists of two steps.

{\it Step 1.}  In Section~\ref{ess spectrum} below we will show that
the graph norm of $\Delta_{\lambda,r}$ is an equivalent norm in
$H^2_0(\mathcal G)$. This immediately implies that
 $\mathrm
D(\Delta_{\lambda,r})=H^2_0(\mathcal G)$ as the set
$C_0^\infty(\overline{\mathcal G})$ is dense in both spaces, cf.
Proposition~\ref{sectorial}.1. We will also localize the essential
spectrum $\sigma_{ess}(\Delta_{\lambda,r})$ of the unbounded
m-sectorial operator $\Delta_{\lambda,r}$ in $L^2(\mathcal G)$. As
is well-known, the spectrum $\sigma(\Delta_\Omega)$ consists of
infinitely many positive isolated eigenvalues. It turns out that
$\sigma_{ess}(\Delta_{\lambda,r})$ consists of an infinite number of
rays emanating from every point $\nu\in\sigma(\Delta_\Omega)$, cf.
Figure~\ref{fig5}.
\begin{figure}[h]
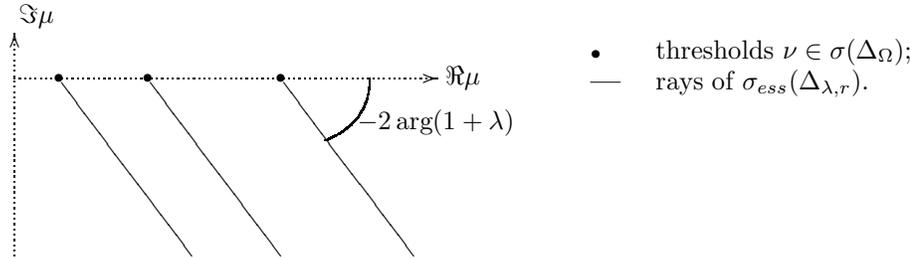

\[
\xy (0,0)*{\xy0;/r.28pc/:{\ar@{.>}(0,10);(50,10)*{\ \Re \mu}};
{\ar@{.>}(0,-10);(0,15)}; (2,17)*{\ \Im \mu};
{\ar@{-}(5,10)*{\scriptstyle\bullet};(20,-10)};{\ar@{-}(15,10)*{\scriptstyle\bullet};(30,-10)};
{(30,10)*{\scriptstyle\bullet};(45,-10)*{} **\dir{-}};
(40,10);(35,3)**\crv{(40,5)*{\quad\quad\quad\quad\quad
-2\arg(1+\lambda)}};
\endxy};
(65,0)*{\xy (20,20)*{
\begin{array}{ll}
     {\scriptstyle\bullet } & \text{ thresholds $\nu\in\sigma(\Delta_\Omega)$;} \\
    \text{\bf---} & \text{ rays of $\sigma_{ess}(\Delta_{\lambda,r})$.}
\end{array}};
\endxy};
\endxy
\]
\caption{Essential spectrum of the m-sectorial operator
$\Delta_{\lambda,r}$ for $\Im\lambda>0$.}\label{fig5}
\end{figure}
By definition $\sigma(\Delta_\Omega)$ is the set of thresholds of
the Dirichlet Laplacian $\Delta$. As $\lambda$ varies, the ray
$\{\mu\in \Bbb C:\arg(\mu-\nu)=-2\arg (1+\lambda)\}$ of the
essential spectrum $\sigma_{ess}(\Delta_{\lambda,r})$ rotates about
the threshold $\nu\in\sigma(\Delta_\Omega)$ and sweeps the sector
$\{\mu\in\Bbb C:|\arg (\mu-\nu)|<2\alpha\}$. In order to avoid
several repetitions we organized the paper so that the proofs of
these results in some greater generality are postponed to
Section~\ref{ess spectrum}; here we take these results for granted.

Recall that $\mu$ is said to be a point of the essential spectrum
$\sigma_{ess}(A)$ of a closed unbounded operator $A$ in the space
$L^2(\mathcal G)$ with the domain $H^2_0(\mathcal G)$, if the
bounded operator $A-\mu: H^2_0(\mathcal G)\to L^2(\mathcal G)$ is
not Fredholm (a linear bounded operator between two Banach spaces is
Fredholm, if its kernel and cokernel are finite dimensional, and its
range is closed). Since the operator  $\Delta_{\lambda,r}$ is
m-sectorial, there exists a regular point of $\Delta_{\lambda,r}$ in
the simply connected set $\mathbb
C\setminus\sigma_{ess}(\Delta_{\lambda,r})$. Therefore
$$
\mathbb
C\setminus\sigma_{ess}(\Delta_{\lambda,r})\ni\mu\mapsto\Delta_{\lambda,r}-\mu:H^2_0(\mathcal
G)\to L^2(\mathcal G)
$$
is a Fredholm holomorphic operator function,
e.g.~\cite[Appendix]{KozlovMaz`ya}. Recall that the spectrum of a
Fredholm holomorphic operator function consists of isolated
eigenvalues of finite algebraic multiplicity, e.g.~\cite[Proposition
A.8.4]{KozlovMaz`ya}. As a consequence, the resolvent
\begin{equation}\label{res}\mathbb
C\setminus\sigma_{ess}(\Delta_{\lambda,r})\ni\mu\mapsto(\Delta_{\lambda,r}-\mu)^{-1}:L^2(\mathcal
G)\to H^2_0(\mathcal G)
\end{equation}
 is a meromorphic operator function.

 For  $f,g\in L^2(\mathcal G)$ and a sufficiently large $r>0$ we
define the quadratic form
$$
(f,g)_{\lambda,r}=\int_{\mathcal G}f \overline{g}\sqrt{\det\mathsf
e_{\lambda,r}}\,d\zeta\,d\eta,\quad \lambda\in\mathcal D_\alpha.
$$
This   form is bounded in $L^2(\mathcal G)$. Indeed, thanks to the
estimate~\eqref{cont+} on $\mathsf e^{-1}_{\lambda,r}$ we have
$0<c_1\leq|\det\mathsf e_{\lambda,r}(\zeta,\eta)|\leq c_2$ uniformly
in $\lambda\in\mathcal D_\alpha$ and
$(\zeta,\eta)\in\overline{\mathcal G}$.

Assume that $\lambda\in\mathcal D_\alpha$ is real. Then the form
$(\cdot,\cdot)_{\lambda,r}$ is the inner product  induced on
$\mathcal G$ by the  metric $\mathsf e_{\lambda,r}$, the norm
$\sqrt{(f,f)_{\lambda,r}}$ is equivalent to the norm
$\|f;L^2(\mathcal G)\|$, and $\Delta_{\lambda,r}$ is the
Laplace-Beltrami operator on $(\mathcal G,\mathsf e_{\lambda,r})$.
We have
\begin{equation}\label{RG}(\Delta
-\mu)u=\bigl((\Delta_{\lambda,r}-\mu)(u\circ\vartheta_{\lambda,r})\bigr)\circ\vartheta_{\lambda,r}^{-1}\quad
\forall u\in C_0^\infty(\overline{\mathcal G}).
\end{equation}
Assume that $\mu$ is not in the sector of $\Delta_{\lambda,r}$. Then
$(\Delta_{\lambda,r}-\mu)^{-1}$ is a bounded operator  and we can
rewrite~\eqref{RG} in the form
\begin{equation}\label{++}
(\Delta
-\mu)^{-1}f=\bigl((\Delta_{\lambda,r}-\mu)^{-1}(f\circ\vartheta_{\lambda,r})\bigr)\circ\vartheta_{\lambda,r}^{-1},
\end{equation}
where $f$ is in the set $\{f=(\Delta-\mu)u: u\in
C^\infty_0(\overline{\mathcal G})\}$. This set is dense in
$L^2(\mathcal G)$, because $C_0^\infty(\overline{\mathcal G})$ is
dense in $H^2_0(\mathcal G)$, and the operator
$\Delta-\mu:H^2_0(\mathcal G)\to L^2(\mathcal G)$ yields an
isomorphism. It is clear that
$(f\circ\vartheta_{\lambda,r},f\circ\vartheta_{\lambda,r})_{\lambda,r}=(f,f)$.
As a consequence, the (real) scaling $f\mapsto
f\circ\vartheta_{\lambda,r}$ realizes an isomorphism in
$L^2(\mathcal G)$, and the equality~\eqref{++} extends by continuity
to all $f\in L^2(\mathcal G)$. Taking the inner product  of the
equality~\eqref{++} with $g\in L^2(\mathcal G)$, and passing to the
variables $(\tilde\zeta,\tilde\eta)=\vartheta_\lambda(\zeta,\eta)$
in the right hand side, we obtain
\begin{equation}\label{h6}
\bigl((\Delta-\mu)^{-1}f,g
\bigr)=\bigl((\Delta_{\lambda,r}-\mu)^{-1}(f\circ\vartheta_{\lambda,r}),g\circ\vartheta_{\overline{\lambda},r}
\bigr)_{\lambda,r}.
\end{equation}Now we
assume that $f,g\in\mathcal H_\alpha(\mathcal G)$. Then
$f\circ\vartheta_{\lambda,r}$ and $g\circ\vartheta_{\lambda,r}$ are
$L^2(\mathcal G)$-valued analytic functions of $\lambda$ in the disk
$\mathcal D_\alpha$, see  Corollary~\ref{dense}.3. This together
with Proposition~\ref{sectorial}.4 implies that the right hand side
of~\eqref{h6} extends by analyticity from $\lambda\in \mathcal
D_\alpha\cap\mathbb R$ to all $\lambda\in\mathcal D_\alpha$. The
right hand side of~\eqref{h6} extends from all $\mu$ outside of the
sector of $\Delta_{\lambda,r}$ to a meromorphic function of $\mu\in
\Bbb C\setminus\sigma_{ess}(\Delta_{\lambda,r})$. In particular, for
all $\lambda\in\mathcal D_\alpha\setminus\Bbb R$ we have $\mu_0\in
\mathbb C\setminus\sigma_{ess}(\Delta_{\lambda,r})$, cf.
Figure~\ref{fig5}. Here $\mu_0$ is the same as in the formulation of
the theorem.

Now we are in position to prove assertion~1. Consider the projection
$$
\mathsf P=\operatorname{s-}\!\lim_{\epsilon\downarrow 0} i\epsilon
(\Delta-\mu_0-i\epsilon)^{-1}
$$
  onto the eigenspace of
the selfadjoint operator $\Delta$. Suppose, by contradiction, that
the resolvent~\eqref{res} has a pole at the point $\mu_0\in \mathbb
C\setminus\sigma_{ess}(\Delta_{\lambda,r})$. By
Corollary~\ref{dense}.2 there exist $f$ and $g$ in the space
$\mathcal H_\alpha(\mathcal G)$ such that $\mu_0$ is a pole of the
right hand side of~\eqref{h6}. The equality~\eqref{h6} implies that
$(\mathsf P f,g)\neq 0$, and thus $\ker(\Delta-\mu_0)\neq\{0\}$.
This is a contradiction. Assertion~1 is proven.

{\it Step 2.} We need to show that for any $\varrho\in
C_c^\infty(\overline{\mathcal G})$ the operator
$\varrho(\Delta-\mu_0-i\epsilon)^{-1}$
 tends to some
limits in the space of bounded operators  $\mathscr B(\mathscr
H_\alpha(\mathcal G), H^2_0(\mathcal G))$ as $\epsilon\downarrow 0$
and $\epsilon\uparrow 0$. We take a sufficiently large
$r=r(\varrho)>0$ such that $\supp\varrho\subset\overline{\mathcal
G_r}$. Then $\varrho\circ\vartheta_{\lambda,r}=\varrho$ for all
$\lambda\in\mathcal D_\alpha$.  Now we can pass from~\eqref{++} to
the equality
\begin{equation}\label{q1}
\varrho(\Delta
-\mu)^{-1}f=\varrho(\Delta_{\lambda,r}-\mu)^{-1}(f\circ\vartheta_{\lambda,r}).
\end{equation}
For $f\in \mathcal H_\alpha(\mathcal G)$ the equality~\eqref{q1}
extends by analyticity to all $\lambda\in\mathcal D_\alpha$.
Consider, for instance, the case $\Im\lambda>0$ (the case
$\Im\lambda<0$ is similar). Then the upper half-plane $\Bbb
C^+=\{\mu\in\Bbb C: \Im\mu>0\}$ and a complex neighborhood of the
point $\mu_0$ do not contain points of
$\sigma_{ess}(\Delta_{\lambda,r})$, cf. Figure~\ref{fig5}. Therefore
the right hand side of~\eqref{q1} has a meromorphic continuation in
$\mu$ to the union of $\Bbb C^+$ and a complex neighborhood of
$\mu_0$. Hence the left hand side of~\eqref{q1} has the same
meromorphic continuation. Clearly, a pole at $\mu_0$ may only appear
due to a pole of the resolvent~\eqref{res} at $\mu_0$, but it is a
regular point by assertion 1. Since $\varrho$ is an arbitrary smooth
function supported in $\mathcal G_{r}$ this proves assertion 3. In
order to prove assertion 2 it remains to note that
$$
\begin{aligned}
\|\lim_{\epsilon \downarrow 0}\varrho (\Delta
-\mu_0-i\epsilon)^{-1}f; H^2_0(\mathcal G)
\|=\|\varrho(\Delta_{\lambda,r}-\mu_0)^{-1}(f\circ\vartheta_{\lambda,r});
H^2_0(\mathcal G)\|
\\
\leq C({\varrho})\|(\Delta_{\lambda,r}-\mu_0)^{-1};L^2(\mathcal
G)\to H^2_0(\mathcal G) \|\|f\circ\vartheta_{\lambda,r};
L^2(\mathcal G)\|
\\
\leq C(\varrho,r,\mu_0)\|f; \mathcal H_\alpha(\mathcal G)\|.
\end{aligned}
$$
In the last inequality we used Corollary~\ref{dense}.1 and assertion
1.
\end{proof}

In the following two remarks we collect some results that can be
obtained by methods  developed in the proof of Theorem~\ref{LA}.
Although these results are not used in this paper, they provide
additional insights of the problem.
\begin{remark}\label{ABCS}  On the basis of the equality~\eqref{h6},
Corollary~\ref{dense}, and the description of
$\sigma_{ess}(\Delta_{\lambda, r})$ for $\lambda\in\mathcal
D_\alpha$, one can develop  an analog of the celebrated
Aguilar-Balslev-Combes-Simon theory of resonances~\cite{Cycon,Hislop
Sigal,Hunziker,Simon Reed iv}. We announce some results below, for
the proof we refer to~\cite{KalvinDL}.
\begin{enumerate}
\item The selfadjoint Dirichlet Laplacian $\Delta$ in $\mathcal G$ has
no singular continuous spectrum, its eigenvalues can accumulate only
at thresholds $\nu\in\sigma(\Delta_\Omega)$. (Examples of
accumulating eigenvalues can be found e.g. in~\cite{Edward}.)
\item The spectrum $\sigma(\Delta_{\lambda,r})$ of the m-sectorial
operator $\Delta_{\lambda,r}$ does not depend on the choice of the
scaling function satisfying~\eqref{ab1}--\eqref{ab3}. Moreover, the
spectrum $\sigma(\Delta_{\lambda,r})$ lies in the half-plain $
\overline{\Bbb C^+}$ in the case $\Im \lambda\geq 0$ and
$\sigma(\Delta_{\lambda,r})\subset \overline{\Bbb C^-}$ in the case
$\Im\lambda\leq 0$, where $\Bbb C^\pm=\{\mu\in\Bbb C:\Im\mu\gtrless
0\}$ and $\lambda\in \mathcal D_\alpha$.

\item A  point $\mu\in \Bbb
R\setminus\sigma(\Delta_\Omega)$ is an eigenvalue of $\Delta$ if and
only if it is an isolated eigenvalue of $\Delta_{\lambda,r}$, where
$\lambda\in \mathcal D_\alpha\setminus\Bbb R$ and $r>0$ is
sufficiently large.

\item  The resolvent
matrix elements $((\Delta-\mu)^{-1}f,g)$, where $f,g\in\mathcal
H_\alpha(\mathcal G)$, have meromorphic continuations from the
physical sheet $\Bbb C\setminus\sigma_{ess}(\Delta)$ across
$\sigma_{ess}(\Delta)$ to the set  $\Bbb
C\setminus\sigma_{ess}(\Delta_{\lambda, r})$. Moreover, $\mu$ is a
pole of the continuation for some $f,g\in\mathcal H_\alpha(\mathcal
G)$ if and only if it is an isolated eigenvalue of
$\Delta_{\lambda,r}$. The non-real isolated eigenvalues of
$\Delta_{\lambda,r}$ are naturally identified with resonances of
$\Delta$.

\item As $\lambda$ changes continuously in the disk $\mathcal
D_\alpha$, an isolated eigenvalue of $\Delta_{\lambda, r}$ survive
while it is not covered by one of the rotating rays of
$\sigma_{ess}(\Delta_{\lambda, r})$.
\end{enumerate}
\end{remark}
\begin{remark}\label{RAC} The argument of the second step in the proof
of Theorem~\ref{LA} allows also to see that the resolvent
$(\Delta-\mu)^{-1}: \mathcal H_\alpha(\mathcal G)\to
H^2_{0,\operatorname{loc}}(\mathcal G)$ has a meromorphic
continuation from the physical sheet $\Bbb
C\setminus\sigma_{ess}(\Delta)$ across the intervals $(\nu_-,\nu_+)$
between the neighboring thresholds $\nu_\pm\in\sigma(\Delta_\Omega)$
to a Riemann surface. The surface consists of the physical sheet
$\Bbb C\setminus\sigma_{ess}(\Delta)$ of the Dirichlet Laplacian and
an infinite number of the sectors $\{\mu\in \Bbb
C:0>\arg(\mu-\nu_-)>-2\alpha\}$ attached to $\Bbb C^+\subset \Bbb
C\setminus\sigma_{ess}(\Delta)$ and of the sectors $\{\mu\in \Bbb
C:0<\arg(\mu-\nu_-)<2\alpha\}$ attached to $\Bbb C^-\subset\Bbb
C\setminus\sigma_{ess}(\Delta)$ along the intervals $(\nu_-,
\nu_+)$. Indeed, for any $\varrho\in C_c^\infty(\overline{\mathcal
G})$ we can take a sufficiently large $r=r(\varrho)>0$ such that
$\supp\varrho\subset\overline{\mathcal G_r}$. As $\lambda$ varies in
$\mathcal D_\alpha\cap\Bbb C^+$ (resp. in $\mathcal D_\alpha\cap\Bbb
C^-$)  the strip between the neighboring rays $\{\mu\in\Bbb
C:\arg(z-\nu_\pm)=-2\arg (1+\lambda)\}$ of
$\sigma_{ess}(\Delta_{\lambda,r})$ sweeps the sector $\{\mu\in \Bbb
C:0>\arg(\mu-\nu_-)>-2\alpha\}$ (resp. the sector $\{\mu\in \Bbb
C:0<\arg(\mu-\nu_-)<2\alpha\}$) and the right hand side
of~\eqref{q1} provides the left hand side with a meromorphic
continuation to the strip. The poles of the continuation are
resonances of the Dirichlet Laplacian $\Delta$~\cite{Zworski}.
\end{remark}

The results listed in Remarks~\ref{ABCS} and~\ref{RAC} are new and
might be of their own interest. Traditionally, when studying
Laplacians in the waveguide-type of geometry, one imposes  more or
less restrictive assumptions on the rate of convergence of the
metric $\mathsf g$ on $(0,\infty)\times\overline{\Omega}$ to its
limit at infinity; see, e.g.,~\cite{chr
ST,DES,DEM,Edward,FroHislop,Guillope,IKL,Melrose,MelroseScat,Mueller}.
Contrastingly, our assumptions on the diffeomorphism $\varkappa$
allow for arbitrarily slow convergence of the metric $\mathsf
g=\varkappa^*\mathsf e$ to the Euclidean metric $\mathsf e$ at
infinity, see Section~\ref{sec3}. As a substitution for the
assumptions on the rate of
 convergence of the metric $\mathsf g$ at infinity we use assumptions on the
analytic regularity of the diffeomorphism $\varkappa$.

Let us also mention here the paper~\cite{Plamenevskii}, where
general elliptic problems whose coefficients slowly converge to
their limits at infinity are considered. It is shown that any finite
accumulation point of eigenvalues corresponding to exponentially
decaying eigenfunctions is a threshold, these accumulations are
characterized in terms of some non-classical ``augmented scattering
matrices.'' An additional investigation on decay of eigenfunctions
at infinity is required in order to say whether these results
 describe accumulations of eigenvalues of the
Dirichlet Laplacian $\Delta$ or not. This goes beyond the scope of
the present paper, we refer to~\cite{KalvinExp}.

\section{Exponential decay of solutions in infinite PMLs}\label{s5} In this
section we prove the following theorem.
\begin{theorem}\label{ED}Assume that   $\mu_0\in\Bbb
R\setminus\sigma(\Delta_\Omega)$ is not an eigenvalue of the
selfadjoint Dirichlet Laplacian $\Delta$ in $L^2(\mathcal G)$, the
scaling parameter $\lambda\in\mathcal D_\alpha$ is not real, and
\begin{equation}\label{interval}
0\leq\beta<min_{\nu\in\sigma(\Delta_\Omega)}|\Im\{
(1+\lambda)\sqrt{\mu_0-\nu}\}|.
\end{equation}
 Let $\mathsf s(\zeta,\eta)$ stand for
the scaling function $\Bbb R_+\times\overline{\Omega}\ni(x,y)\mapsto
s(x)$ written in the coordinates $(\zeta,\eta)\in\overline{\mathcal
C}$ and extended to $\overline{\mathcal G}$ by zero;
see~\eqref{ab1}--\eqref{ab3}. Then for all sufficiently large $r>0$
and all $\mathcal F\in L^2(\mathcal G)$ satisfying $e^{\beta \mathsf
s}\mathcal F\in L^2(\mathcal G)$ the estimate
\begin{equation}\label{est++}
\|e^{\beta \mathsf s}(\Delta_{\lambda,r}-\mu_0)^{-1}\mathcal
F;H^2_0(\mathcal G)\|\leq C\|e^{\beta \mathsf s}\mathcal F;
L^2(\mathcal G)\|
\end{equation}
is valid with a constant $C$ independent of $\mathcal F$.
\end{theorem}

Theorem~\ref{ED} together with Theorem~\ref{LA} shows that under the
additional assumption $e^{\beta \mathsf
s}(f\circ\vartheta_{\lambda,r})\in L^2(\mathcal G)$ on $f\in
\mathcal H_\alpha(\mathcal G)$ infinite PMLs absorb outgoing or
incoming solutions (depending on the sign of $\Im\lambda$) so
effectively that the function $u_{\lambda,r}$ in Theorem~\ref{ED}.3
exponentially decays at infinity in the mean.
\begin{proof} Consider the conjugated operator $e^{\beta\mathsf s}\Delta_{\lambda,r}e^{-\beta\mathsf s}$ as an unbounded operator in
$L^2(\mathcal G)$ with the domain $C_0^\infty(\overline{\mathcal
G})$. With this operator we associate the quadratic form $
q_{\lambda,r}^\beta[u,u]=(e^{\beta\mathsf
s}\Delta_{\lambda,r}e^{-\beta\mathsf s}u,u)_{\lambda,r}$. Observe
that
$$
\begin{aligned}
q_{\lambda,r}^\beta[u,u]- &
q_{\lambda,r}[u,u]=-\beta^2\int_{\mathcal G} \bigl\langle\mathsf
e_{\lambda,r}^{-1} u\nabla_{\zeta \eta}\mathsf s,  u\nabla_{\zeta
\eta}\mathsf s \bigr\rangle\,d\zeta\,d\eta
\\
& -\beta\int_{\mathcal G} \bigl\langle\bigl({\det \mathsf
e_{\lambda,r} }\bigr)^{1/2}\mathsf e_{\lambda,r}^{-1} u\nabla_{\zeta
\eta}\mathsf s, \nabla_{\zeta \eta}\bigl({\det \mathsf
e_{\overline{\lambda},r} }\bigr)^{-1/2}
u\bigr\rangle\,d\zeta\,d\eta\\
 & +\beta\int_{\mathcal G}
\bigl\langle\mathsf e_{\lambda,r}^{-1} \nabla_{\zeta \eta}u,
u\nabla_{\zeta \eta}\mathsf s \bigr\rangle\,d\zeta\,d\eta,
\end{aligned}
$$
where $q_{\lambda,r}$ is the same as in~\eqref{qfD}. Since the right
hand side depends linearly on $\nabla_{\zeta \eta}u$, similarly to
the  the proof of Lemma~\ref{r bound} one can deduce
$$
\bigl|q_{\lambda,r}^\beta[u,u]-q_{\lambda,r}[u,u]\bigr|\leq \epsilon
|( -\nabla_{\zeta \eta}\cdot\mathsf e_{\lambda,r}^{-1} \nabla_{\zeta
\eta}u,u)
 |+C\epsilon^{-1}\|u; L^2(\mathcal G)\|^2
$$
with an arbitrary small $\epsilon>0$ and a constant $C$ independent
of $u\in C_0^\infty(\overline{\mathcal G})$. This estimate together
with~\eqref{i1},~\eqref{i2}, and Lemma~\ref{l1} implies that
 for all sufficiently large $r>0$ we
have
\begin{equation}\label{i1+}
 |\arg\bigl( q^\beta_{\lambda,r} [u,u]+\gamma \|u; L^2(\mathcal G)\|^2\bigr)|\leq \varphi
\end{equation}
with some angle $\varphi<\pi/2$ and $\gamma>0$, which are
independent of $u\in C_0^\infty(\overline{\mathcal G})$. Therefore
$e^{\beta\mathsf s}\Delta_{\lambda,r} e^{-\beta\mathsf s}$ with the
domain $C_0^\infty(\overline{\mathcal G})$ is a densely defined
sectorial operator in $L^2(\mathcal G)$. Let $\mathrm
D(e^{\beta\mathsf s}\Delta_{\lambda,r} e^{-\beta\mathsf s})$ be the
domain of its m-sectorial Friedrichs extension~\cite[Chapter
VI.2]{Kato}. The Friedrichs extension will also be denoted by
$e^{\beta\mathsf s}\Delta_{\lambda,r} e^{-\beta\mathsf s}$.

As in the proof of Proposition~\ref{sectorial}.1 we conclude that
$C_0^\infty(\overline{\mathcal G})$ is a core of the m-sectorial
operator $e^{\beta\mathsf s}\Delta_{\lambda,r} e^{-\beta\mathsf s}$.
In Section~\ref{ess spectrum} we will show that the graph norm of
$e^{\beta\mathsf s}\Delta_{\lambda,r} e^{-\beta\mathsf s}$ is
equivalent to the norm in $H^2_0(\mathcal G)$; hence $\mathrm
D(e^{\beta\mathsf s}\Delta_{\lambda,r} e^{-\beta\mathsf
s})=H^2_0(\mathcal G)$. Furthermore, we will localize the essential
spectrum $\sigma_{ess}(e^{\beta\mathsf s}\Delta_{\lambda,r}
e^{-\beta\mathsf s})$ of the m-sectorial operator $e^{\beta\mathsf
s}\Delta_{\lambda,r} e^{-\beta\mathsf s}$,  see
Proposition~\ref{ess}. It turns out that the essential spectrum
consists of an infinite number of parabolas, see Figure~\ref{fig++}.
In the case $\beta=0$ the parabolas collapse to the dashed rays
originating from the thresholds $\nu\in\sigma(\Delta_\Omega)$ and we
obtain the essential spectrum $\sigma_{ess}(\Delta_{\lambda,r})$.

The m-sectorial operator $e^{\beta\mathsf s}\Delta_{\lambda,r}
e^{-\beta\mathsf s}$
 defines the Fredholm holomorphic operator function
$$
\mu\mapsto e^{\beta\mathsf s}\Delta_{\lambda,r} e^{-\beta\mathsf
s}-\mu: H^2_0(\mathcal G)\to L^2(\mathcal G)
$$
 on the simply connected subset of $\Bbb
C\setminus\sigma_{ess}(e^{\beta\mathsf s}\Delta_{\lambda,r}
e^{-\beta\mathsf s})$ containing an infinite part of the real
negative semiaxis (regular points of $e^{\beta\mathsf
s}\Delta_{\lambda,r} e^{-\beta\mathsf s}$).
Condition~\eqref{interval} on $\beta$ guarantees that the point
$\mu_0$ is in this simply connected subset. As the spectrum of a
Fredholm holomorphic operator function consists of isolated
eigenvalues of finite multiplicity, $\mu_0$ is a regular point or an
eigenvalue of $e^{\beta\mathsf s}\Delta_{\lambda,r} e^{-\beta\mathsf
s}$. The inclusion $\Psi\in \ker (e^{\beta\mathsf
s}\Delta_{\lambda,r} e^{-\beta\mathsf s}-\mu_0)$ implies
$e^{-\beta\mathsf s}\Psi\in\ker (\Delta_{\lambda,r} -\mu_0)$, and
hence $\Psi\equiv 0$ by Theorem~\ref{LA}.1. Thus the operator
$e^{\beta\mathsf s}\Delta_{\lambda,r} e^{-\beta\mathsf s}-\mu_0$
yields an isomorphism between the spaces $ H^2_0(\mathcal G)$ and
$L^2(\mathcal G)$. This together with the equality
$$
(e^{\beta\mathsf s}\Delta_{\lambda,r} e^{-\beta\mathsf
s}-\mu_0)^{-1}e^{\beta \mathsf s}\mathcal F=e^{\beta\mathsf
s}(\Delta_{\lambda,r}-\mu_0)^{-1} \mathcal F,\quad e^{\beta\mathsf
s}\mathcal F\in L^2(\mathcal G),
$$
justifies the estimate~\eqref{est++}.
\begin{figure}
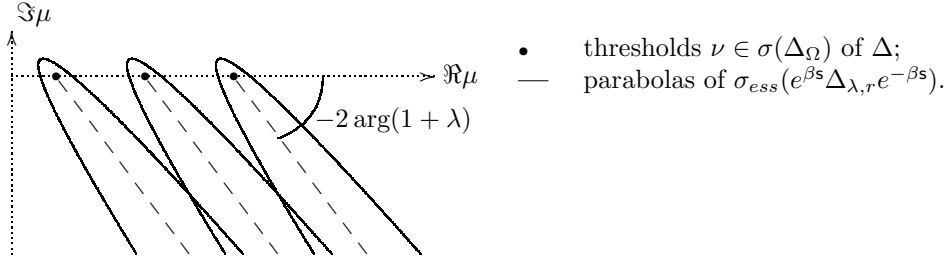
\centering
\[
\xy (0,0)*{\xy0;/r.28pc/:{\ar@{.>}(0,10);(50,10)*{\ \Re \mu}};
{\ar@{.>}(0,-10);(0,15)}; (2,17)*{\ \Im \mu};
{\ar@{--}(5,10)*{\scriptstyle\bullet};(20,-10)};(14,-10);(26,-10)**\crv{(-13,34)};
{\ar@{--}(15,10)*{\scriptstyle\bullet};(30,-10)};(24,-10);(36,-10)**\crv{(-3,34)};
{(25,10)*{\scriptstyle\bullet};(40,-10)*{} **\dir{--}};
(35,10);(30,3)**\crv{(35,5)*{\quad\quad\quad\quad\quad\
-2\arg(1+\lambda)}};(34,-10);(46,-10)**\crv{(7,34)};
\endxy};
(65,0)*{\xy (20,20)*{
\begin{array}{ll}
     {\scriptstyle\bullet } & \text{ thresholds $\nu\in\sigma(\Delta_\Omega)$ of $\Delta$;} \\
    \text{\bf---} & \text{ parabolas of $\sigma_{ess}(e^{\beta \mathsf
s}\Delta_{\lambda,r}e^{-\beta \mathsf s})$.}
\end{array}};
\endxy};
\endxy
\]
\caption{Essential spectrum of the conjugated operator $e^{\beta
\mathsf s}\Delta_{\lambda,r}e^{-\beta \mathsf s}$  for
$\Im\lambda>0$ and $\beta\gtrless 0$.}\label{fig++}
\end{figure}
\end{proof}
\section{Localization of the essential spectrum}\label{ess spectrum}
In this section we  localize the essential spectrum
$\sigma_{ess}(e^{\beta\mathsf s}\Delta_{\lambda,r}e^{-\beta\mathsf
s})$ of the  m-sectorial operator $e^{\beta\mathsf
s}\Delta_{\lambda,r}e^{-\beta\mathsf s}$ in $L^2(\mathcal G)$ with
parameters $\lambda\in\mathcal D_\alpha$ and $\beta\in \Bbb R$; here
 $\mathsf s$ is the same as in Theorem~\ref{ED}. In particular,  in the case $\beta=0$ we find $\sigma_{ess}(\Delta_{\lambda,r})$.
 We also prove that $\mathrm D(e^{\beta\mathsf
s}\Delta_{\lambda,r}e^{-\beta\mathsf s})=H^2_0(\mathcal G)$. In
other words, we show that the information on the essential spectrum
and the domain of $e^{\beta\mathsf
s}\Delta_{\lambda,r}e^{-\beta\mathsf s}$ we used in the proofs of
Theorem~\ref{LA} and Theorem~\ref{ED} is correct.

Let us note that for fixed
 $\lambda$ and $\beta$ the spectrum
$\sigma_{ess}(e^{\beta\mathsf s}\Delta_{\lambda,r}e^{-\beta\mathsf
s})$ depends only on the behavior of the scaling function $\mathsf
s$ and the matrix $\mathsf e_{\lambda,r}$ outside any compact region
of $\mathcal G$. In order to control $\sigma_{ess}(e^{\beta\mathsf
s}\Delta_{\lambda,r}e^{-\beta\mathsf s})$ we imposed the
condition~\eqref{ab3}.

\begin{proposition}\label{ess} Assume that $\lambda\in\mathcal D_\alpha$, $\beta\in\Bbb R$, and $r>0$ is sufficiently
large. Then the following assertions hold.
\begin{itemize}
\item[1.] The Hilbert spaces
$\mathrm D(e^{\beta\mathsf s}\Delta_{\lambda,r}e^{-\beta\mathsf s})$
and  $H^2_0(\mathcal G)$ are coincident and their norms are
equivalent.

\item[2.] The bounded operator
\begin{equation}\label{coc}
e^{\beta\mathsf s}\Delta_{\lambda,r}e^{-\beta\mathsf s}-\mu:
H^2_0(\mathcal G)\to L^2(\mathcal G)
\end{equation}
 is not  Fredholm (or, equivalently, $\mu\in\sigma_{ess}(e^{\beta\mathsf s}\Delta_{\lambda,r}e^{-\beta\mathsf s})$ ) if and only if the parameters
$\mu$, $\lambda$, and $\beta$  meet the condition
\begin{equation}\label{eq9}
\nu-\mu=(1+\lambda)^{-2}(\beta+i\xi)^2\text{ for some
}\nu\in\sigma(\Delta_\Omega) \text{ and } \xi\in\Bbb R.
\end{equation}
\end{itemize}
\end{proposition}
\begin{proof}
The proof is essentially based on methods of the theory  of elliptic
non-homogeneous boundary value
problems~\cite{KozlovMaz`ya,KozlovMazyaRossmann,MP2,Lions Magenes}.
We will rely on the following  lemma  due to Peetre, see
e.g.~\cite[Lemma~5.1]{Lions Magenes},
\cite[Lemma~3.4.1]{KozlovMazyaRossmann} or~\cite{Peetre}:
\begin{itemize}
\item[]{\it Let $\mathcal X,\mathcal Y$ and $\mathcal Z$ be Banach spaces, where $\mathcal X$ is compactly embedded into $\mathcal Z$. Furthemore, let $\mathcal L$ be a linear bounded operator from $\mathcal X$ to $\mathcal Y$. Then the next two assertions are equivalent: (i) the range of $\mathcal L$ is closed in $\mathcal Y$ and $\dim \ker \mathcal L<\infty$, (ii) there exists a constant $C$, such that
    \begin{equation}\label{1}
    \|u;{\mathcal X}\|\leq C(\|\mathcal L u;{\mathcal Y}\|+\|u;\mathcal Z\|)\quad \forall u\in \mathcal X.
    \end{equation}}
\end{itemize}
Below we assume that $\mu$, $\lambda$, and $\beta$ does not meet the
condition~\eqref{eq9} and establish the coercive estimate
 \begin{equation}\label{peetre}
\|u;H^2_0(\mathcal G)\|\leq C(\|(e^{\beta\mathsf
s}\Delta_{\lambda,r}e^{-\beta\mathsf s}-\mu)u; L^2(\mathcal
G)\|+\|\mathsf w u; L^2(\mathcal G)\|)\quad \forall u\in
H^2_0(\mathcal G).
\end{equation}
Here $\mathsf w\in C^\infty(\overline{\mathcal G}) $ is a positive
rapidly decreasing at infinity  weight, such that the embedding of
$H^2_0(\mathcal G)$ into the weighted space $L^2(\mathcal G;\mathsf
w)$ with the norm $\|\mathsf w \cdot; L^2(\mathcal G)\|$ is compact.
Note that~\eqref{peetre} is an estimate of type~\eqref{1} for the
operator~\eqref{coc}.

 The strongly
elliptic differential operator $e^{\beta\mathsf
s}\Delta_{\lambda,r}e^{-\beta\mathsf s}$ endowed with  the Dirichlet
boundary condition set up a regular elliptic boundary value problem.
Solutions of a regular elliptic boundary value problem satisfy local
coercive estimates, e.g.~\cite{Lions Magenes}
or~\cite{KozlovMazyaRossmann}. Thus we have the local coercive
estimate
\begin{equation}\label{a p}
\|\rho_T u;H^2_0(\mathcal G)\|\leq C(\|\varrho_T(e^{\beta\mathsf
s}\Delta_{\lambda,r}e^{-\beta\mathsf s}-\mu)u; L^2(\mathcal
G)\|+\|\varrho_T u; L^2(\mathcal G)\|).
\end{equation}
Here $\rho_T$ and $\varrho_T $ are smooth compactly supported cutoff
functions in $\mathcal G$ such that $\rho_T(\zeta,\eta)=1$ for
$|\zeta|<T+1$ and $\varrho_T\rho_T=\rho_T$, where  $T$ is a large
fixed number.

Let $\chi_T\in C^\infty(\overline{\mathcal G})$ be another cutoff
function such that $\chi_T(\zeta,\eta)=1$ for $|\zeta|>T$ and
$\chi_T(\zeta,\eta)=0$ for $|\zeta|<T-1$. On the next step we
establish the estimate~\eqref{peetre} with $u$ replaced by  $\chi_T
u$. We will do it in the coordinates $(x,y)\in\Bbb
R_+\times\overline{\Omega}$.

Let $L^2(\Bbb R\times\Omega)$ be the space of functions in the
infinite cylinder $\Bbb R\times\Omega$ with the norm
$\bigl(\int_{\Bbb R}\|\mathsf u(x);
L^2(\Omega)\|^2\,dx\bigr)^{1/2}$. Introduce the Sobolev space
$H^2_0(\Bbb R\times\Omega)$ of functions with zero Dirichlet data on
$\mathbb R\times\partial\Omega$ as the completion of the set
$C_0^\infty(\Bbb R\times\overline{\Omega})$ with respect to the norm
$$
\|\mathsf u;H^2_0(\mathbb R\times\Omega)\|=\Bigl(\sum_{ \ell+|m|\leq
2 }
 \|\partial_x^\ell\partial _y^m \mathsf u;
L^2(\Bbb R\times\Omega)\|^2\Bigr)^{1/2}.
$$
Denote $\mathsf u=(\chi_T u)\circ\varkappa$, where $\varkappa$ is
the diffeomorphism~\eqref{diff}. Let
 \begin{equation}\label{LB_l}
 \triangle_{\lambda,r}=-\bigl(\det \mathsf
 g_{\lambda,r} \bigr)^{-1/2}\nabla_{xy}\cdot \bigl(\det \mathsf
 g_{\lambda,r} \bigr)^{1/2}\mathsf g^{-1}_{\lambda,r}
 \nabla_{xy},\quad \lambda\in\mathcal D_\alpha,
 \end{equation}
 be the operator~$\Delta_{\lambda,r}$ written in the
coordinates $(x,y)$. Here $\mathsf g_{\lambda,r}$ is the
matrix~\eqref{metric} and
$\nabla_{xy}=(\partial_x,\partial_{y_1}\dots
\partial_{y_n})^\top$. Due to our assumptions on $\varkappa$
the estimates $0<\epsilon\leq\det \varkappa'(x,y)\leq 1/\epsilon$
hold uniformly in $(x,y)\in\Bbb R_+\times\overline{\Omega}$. Hence
for some independent of $u\in C_0^\infty(\overline{\mathcal G})$
constants $c_1$, $c_2$, and $c_3$ we have
\begin{equation}\label{www}
\begin{aligned}
&\|\chi_T u;H^2_0(\mathcal G)\|=\|\Delta (\chi_T u);L^2(\mathcal
G)\|+\|\chi_T u; L^2(\mathcal G)\|
\\
&\quad\leq c_1(\|{\triangle_{0,r}}\mathsf u; L^2(\mathbb
R\times\Omega)\| +\|\mathsf u; L^2(\mathbb R\times\Omega)\|)\leq
c_2\|\mathsf u; H_0^2(\Bbb R\times\Omega)\|,
\\
&\|(e^{\beta s}\triangle_{\lambda,r}e^{-\beta s}-\mu)\mathsf u;
L^2(\mathbb R\times\Omega)\|\leq c_3\|(e^{\beta\mathsf
s}\Delta_{\lambda,r}e^{-\beta\mathsf s}-\mu)\chi_T u; L^2(\mathcal
G)\|.
\end{aligned}
\end{equation}
Here the functions $\mathsf u$, $s$, and
$\triangle_{\lambda,r}\mathsf u\equiv(\Delta_{\lambda,r} (\chi_T
u))\circ\varkappa$ are extended from $\Bbb
R_+\times\overline{\Omega}$ to the infinite cylinder $\mathbb
R\times\overline{\Omega}$ by zero, and $\|\triangle_{0,r}\mathsf u;
L^2(\mathbb R\times\Omega)\|\leq C\|\mathsf u;H_0^2(\Bbb
R\times\Omega)\|$ because the coefficients of the Laplacian
$\triangle_{0,r}$ are bounded, cf.~\eqref{LB_l} and~\eqref{nstab}.
As $T$ is large, the function $\mathsf u$ is supported in a small
neighborhood of infinity. Due to the stabilization
condition~\eqref{lim} on $\mathsf g^{-1}_\lambda$ and the
condition~\eqref{ab3} on the scaling function $s$ the coefficients
of the differential operator
$$
e^{\beta\mathsf s}\triangle_{\lambda,r}e^{-\beta\mathsf s}-
\Delta_\Omega+(1 + \lambda)^{-2}(\partial_x+\beta)^2
$$ are small on the
support of $\mathsf u$. As a result we get the estimate
\begin{equation}\label{eq11}
\bigl\|\bigl(e^{\beta\mathsf s}\triangle_{\lambda,r}e^{-\beta\mathsf
s}- \Delta_\Omega+(1 +
\lambda)^{-2}(\partial_x+\beta)^2\bigr)\mathsf u;L^2(\mathbb
R\times\Omega)\bigr\| \leq \epsilon\|\mathsf u; H^2_0(\mathbb
R\times\Omega)\|,
\end{equation}
where $\epsilon$ is small and independent of $u\in
C^\infty_0(\overline{\mathcal G})$; moreover, $\epsilon\to 0$ as
$T\to+\infty$.

Consider the bounded operator
\begin{equation}\label{mo}
\Delta_\Omega-(1+\lambda)^{-2}(\partial_x+\beta)^2-\mu:
H^2_0(\mathbb R\times\Omega)\to L^2(\mathbb R\times\Omega).
\end{equation}
Applying the Fourier transform $\mathscr F_{x\mapsto \xi}$ we pass
from the operator~\eqref{mo} to the selfadjoint Dirichlet Laplacian
$\Delta_\Omega+(1+\lambda)^{-2}(\beta+i\xi)^2-\mu$ in $L^2(\Omega)$.
Since $\mu$, $\lambda$, and $\beta$ do not meet the
condition~\eqref{eq9}, the spectral parameter
$\mu-(1+\lambda)^{-2}(\beta+i\xi)^2$ is outside of the spectrum  of
$\Delta_\Omega$ for all $\xi\in\Bbb R$. Then a known
argument~\cite[Theorem
5.2.2]{KozlovMazyaRossmann},~\cite[Theorem~2.4.1]{KozlovMaz`ya},
which is also used as a part of the proof of Lemma~\ref{al} below,
implies that the operator~\eqref{mo} realizes an isomorphism. In
particular, the estimate
$$
\begin{aligned}
 \|\mathsf u; H_0^{2}(\mathbb R\times\Omega)\| \leq  \mathrm{c} \bigl\|\bigl(\Delta_\Omega-(1+\lambda)^{-2}(\partial_x+\beta)^2 -\mu\bigr) \mathsf u; L^2(\mathbb R\times\Omega)\bigr\|
\end{aligned}
$$
is valid with an independent of $\mathsf u\in H_0^{2}(\mathbb
R\times\Omega)$ constant $\mathrm{c}$. As a consequence of this
estimate and~\eqref{eq11} we obtain
$$
\begin{aligned}
(1&-\epsilon\mathrm{c}) \|\mathsf u; H_0^{2}(\mathbb
R\times\Omega)\|\leq
\mathrm{c}\bigl\|\bigl(\Delta_\Omega-(1+\lambda)^{-2}(\partial_x+\beta)^2
-\mu\bigr) \mathsf u; L^2(\mathbb R\times\Omega)\bigr\|
\\
&-\mathrm{c}\bigl\|\bigl(e^{\beta s}\triangle_{\lambda,r}e^{-\beta
s}- \Delta_\Omega+(1 +
\lambda)^{-2}(\partial_x+\beta)^2\bigr)\mathsf u;L^2(\mathbb
R\times\Omega)\bigr\|\\
& \leq \mathrm{c}\|(e^{\beta s}\triangle_{\lambda,r}e^{-\beta
s}-\mu)\mathsf u; L^2(\Bbb R\times\Omega)\|.
\end{aligned}
$$
 If $T$ is sufficiently large, then $\epsilon \mathrm{c}<1$.
This together with~\eqref{www} gives
\begin{equation}\label{2}
\|\chi_T u;H^2_0(\mathcal G) \|\leq C\|(e^{\beta\mathsf
s}\Delta_{\lambda,r}e^{-\beta\mathsf s}-\mu)\chi_T u; L^2(\mathcal
G)\|,
\end{equation}
where the constant $C=\mathrm{c}(1-\epsilon \mathrm{c})^{-1}c_2 c_3$
is independent of $u\in C_0^\infty(\overline{\mathcal G})$. By
continuity the estimate~\eqref{2} extends to all $u\in
H^2_0(\mathcal G)$.

Now we combine~\eqref{2} with~\eqref{a p}, and arrive at the
estimates
\begin{equation}\label{3}
\begin{aligned}
\|u;H^2_0(\mathcal G)\|\leq \|\chi_T u;H^2_0(\mathcal
G)\|+\|\rho_Tu;H^2_0(\mathcal G)\|
\\
\leq C(\|\chi_T(e^{\beta\mathsf s}\Delta_{\lambda,r}e^{-\beta\mathsf
s}-\mu) u; L^2(\mathcal G)\| +\|[e^{\beta\mathsf
s}\Delta_{\lambda,r}e^{-\beta\mathsf s},\chi_T] u; L^2(\mathcal G)\|
\\
+\|\varrho_T(e^{\beta\mathsf s}\Delta_{\lambda,r}e^{-\beta\mathsf
s}-\mu) u; L^2(\mathcal G)\| +\|\varrho_T u; L^2(\mathcal G)\|)
\\
\leq C(\|(e^{\beta\mathsf s}\Delta_{\lambda,r}e^{-\beta\mathsf
s}-\mu) u; L^2(\mathcal G)\|+\|\varrho_T u; L^2(\mathcal G)\|).
\end{aligned}
\end{equation}
Here we used that $\rho_T=1$ on the support of the commutator
$[e^{\beta\mathsf s}\Delta_{\lambda,r}e^{-\beta\mathsf s},\chi_T]$,
and hence
$$
\|[e^{\beta\mathsf s}\Delta_{\lambda,r}e^{-\beta\mathsf s},\chi_T]
u; L^2(\mathcal G)\|\leq C\|\rho_Tu;H^2_0(\mathcal G)\|.
$$

For an arbitrary positive weight $\mathsf w$ we have $ \|\varrho_T
u; L^2(\mathcal G)\|\leq C\|\mathsf w u; L^2(\mathcal G)\| $ with an
independent of $u\in H^2_0(\mathcal G)$ constant $C$. Thus the
estimate~\eqref{1} is a direct consequence of~\eqref{3}. By the
Peetre's lemma we conclude that the range of the
operator~\eqref{coc} is closed and the kernel is finite-dimensional.

Clearly, the graph norm  $\|u; L^2(\mathcal G)\|+\|e^{\beta\mathsf
s}\Delta_{\lambda,r}e^{-\beta\mathsf s}u\|$ of $u\in
C_0^\infty(\overline{\mathcal G})$ is majorized by
$\|u;H^2_0(\mathcal G)\|$.
 The estimate~\eqref{1} with $\mathsf w\equiv 1$ implies that the norm $\|u;H^2_0(\mathcal G)\|$
 is majorized by the graph norm of $u$. Since the set $C_0^\infty(\overline{\mathcal
 G})$ is dense in
$\mathrm D(e^{\beta\mathsf s}\Delta_{\lambda,r}e^{-\beta\mathsf s})$
and in $H^2_0(\mathcal G)$, this  proves assertion~{\it 1}.

In order to see that the cokernel $\operatorname{coker}
(e^{\beta\mathsf s}\Delta_{\lambda,r}e^{-\beta\mathsf s}-\mu)=\ker
\bigl((e^{\beta\mathsf s}\Delta_{\lambda,r}e^{-\beta\mathsf
s})^*-\overline\mu\bigr)$ of the operator~\eqref{coc} is
finite-dimensional (if $\mu$, $\lambda$, and $\beta$ does not meet
the condition~\eqref{eq9}) we derive the coercive estimate
\begin{equation}\label{4}
\|u;H^2_0(\mathcal G)\|\leq C(\|\bigl((e^{\beta\mathsf
s}\Delta_{\lambda,r}e^{-\beta\mathsf s})^*-\overline\mu\bigr)u;
L^2(\mathcal G)\|+\|\mathsf w u; L^2(\mathcal G)\|)
\end{equation}
for the adjoint $(e^{\beta\mathsf
s}\Delta_{\lambda,r}e^{-\beta\mathsf s})^*$ of the m-sectorial
operator $e^{\beta\mathsf s}\Delta_{\lambda,r}e^{-\beta\mathsf s}$
 and apply the Peetre's lemma. The m-sectorial operator $(e^{\beta\mathsf
s}\Delta_{\lambda,r}e^{-\beta\mathsf s})^*$ corresponds to the
closed densely defined sectorial form
$\overline{q_{\lambda,r}[u,u]}$ with the domain
$\lefteqn{\stackrel{\circ}{\phantom{\,\,>}}}H^1(\mathcal G)$. The
proof of the estimate~\eqref{4} is similar to the proof
of~\eqref{peetre}, we omit it.

We have proved that the operator~\eqref{coc} is Fredholm provided
the condition~\eqref{eq9} is not satisfied. Now we assume that  the
condition~\eqref{eq9} is met, and show that the operator~\eqref{coc}
is not Fredholm.

Let $\chi$ be a smooth cutoff function on the real line, such that
$\chi(x)=1$ for $|x-3|\leq 1$ and $\chi(x)=0$ for $|x-3|\geq 2$.
Consider the functions
\begin{equation}\label{test}
\mathsf u_\ell(x,\mathrm
y)=\chi(x/\ell)\exp\bigl({i(1+\lambda)\sqrt{\mu-\nu}x-\beta
x}\bigr)\Phi(\mathrm y),\quad (x,\mathrm y)\in \mathbb
R\times\Omega,
\end{equation}
where $\Phi$ is an eigenfunction of $\Delta_\Omega$, corresponding
to the eigenvalue $\nu\in \sigma(\Delta_\Omega)$. The exponent
in~\eqref{test} is an oscillating function of $x$.  Straightforward
calculation shows that
\begin{equation}\label{result}
\bigl\|\bigl(\Delta_\Omega -(1+
\lambda)^{-2}(\partial_x+\beta)^2-\mu\bigr) \mathsf u_\ell;
L^2(\mathbb R\times\Omega)\bigr\|\leq C,\quad\|\mathsf u_\ell;
H^2_0(\Bbb R\times\Omega)\|\to\infty
\end{equation}
 as $\ell\to +\infty$. Similarly to~\eqref{eq11} we conclude that
 \begin{equation}\label{5}
\bigl\|\bigl(e^{\beta s}\triangle_{\lambda,r}e^{-\beta s}-
\Delta_\Omega+(1 + \lambda)^{-2}(\partial_x+\beta)^2\bigr)\mathsf
u_\ell;L^2(\mathbb R\times\Omega)\bigr\| \leq \epsilon_\ell\|\mathsf
u_\ell; H^2_0(\Bbb R\times\Omega)\|,
\end{equation}
where $\epsilon_\ell\to 0$ as $\ell\to+\infty$. Let the functions
$u_\ell=\mathsf u_\ell\circ\varkappa^{-1}$ be extended from
$\mathcal C$ to $\mathcal G$ by zero. If, on the contrary, the
operator~\eqref{coc} is Fredholm, then by the Peetre's lemma the
estimate~\eqref{peetre} holds with any weight $\mathsf w$, such that
$H^2_0(\mathcal G)\hookrightarrow L^2(\mathcal G;\mathsf w)$ is a
compact embedding. Without loss of generality we can assume that
$\|\mathsf w u_\ell; L^2(\mathcal G)\|\leq C$ for all $\ell\geq 1$.
After the change of variables $(\zeta,\eta)\mapsto(x,y)$ the
estimate~\eqref{peetre} implies
$$
\|\mathsf u_\ell; H^2_0(\Bbb R\times\Omega)\|\leq C(\|(e^{\beta
s}\triangle_{\lambda,r}e^{-\beta s}-\mu)\mathsf u_\ell; L^2(\Bbb
R\times\Omega)\|+1),
$$
where the function $e^{\beta s}\triangle_{\lambda,r}e^{-\beta
s}\mathsf u_\ell=(e^{\beta\mathsf
s}\Delta_{\lambda,r}e^{-\beta\mathsf s}u_\ell)\circ\varkappa$ is
extended from $\Bbb R_+\times\overline{\Omega}$ to $\Bbb
R\times\overline{\Omega}$ by zero. This together with~\eqref{5}
justifies the estimate
$$
\|\mathsf u_\ell; H^2_0(\Bbb R\times\Omega)\|\leq
C\bigl(\bigl\|\bigl(\Delta_\Omega -(1+
\lambda)^{-2}(\partial_x+\beta)^2-\mu\bigr) \mathsf u_\ell;
L^2(\mathbb R\times\Omega)\bigr\|+1\bigr),
$$
which contradicts~\eqref{result}.
\end{proof}

\section{Problem with finite PMLs}\label{secPML}
Consider the truncated domain $\mathcal G_R$ with piecewise smooth
boundary, see~\eqref{truncated}.  Introduce the Sobolev space
$H^2_0(\mathcal G_R)$ as the completion of the set
$C_0^\infty(\overline{\mathcal G_R})$ in the norm $\|v;
H^2_0(\mathcal G_R)\|=(\sum_{\ell+|m|\leq
2}\|\partial_\zeta^\ell\partial_\eta^m v;L^2(\mathcal
G_R)\|^2)^{1/2}$. In this section we study the problem with finite
PMLs: {\it Given $g\in L^2(\mathcal G_R)$ find a solution $v\in
H_0^2(\mathcal G_R)$ of the equation
\begin{equation}\label{PML}
(\Delta_{\lambda,r}-\mu_0)v=g\text{ in }\mathcal G_R,
\end{equation}
where $R>r$.}
 The next theorem presents a stability result
for this problem.

\begin{theorem}\label{stability} Assume that $\mu_0\in\Bbb
R\setminus\sigma(\Delta_\Omega)$ is not an eigenvalue of the
selfadjoint Dirichlet Laplacian $\Delta$ in $L^2(\mathcal G)$. Take
a sufficiently large $r>0$ and $\lambda\in\mathcal
D_\alpha\setminus\Bbb R$. Then there exists $R_0>r$ such that for
all $R>R_0$ and $g\in L^2(\mathcal G_R)$ the equation~\eqref{PML}
has a unique solution $v\in H^2_0(\mathcal G_R)$.
 Moreover, the estimate
\begin{equation}\label{e+}
\|v; H^2_0(\mathcal G_R)\| \leq C\|g;L^2(\mathcal G_R)\|
\end{equation}
holds with an independent of $R>R_0$ and $g$ constant $C$.
\end{theorem}

 The proof of Theorem~\ref{stability} will be carried out by the
compound expansion technique. This requires construction of an
approximate solution to the equation~\eqref{PML} compounded of
solutions to limit problems. As the first limit problem we take the
problem with infinite PMLs. As the second limit problem we take a
Dirichlet problem in the semi-cylinder $(-\infty,R)\times\Omega$
studied in the next lemma.

\begin{lemma}\label{al} Introduce the weighted Sobolev space
$H^2_{0,\beta}\bigl((-\infty,R)\times\Omega\bigr)$ of functions
satisfying the homogeneous Dirichlet boundary condition as the
completion of the set
$C_0^\infty\bigl((-\infty,R]\times\overline{\Omega}\bigr)$ with
respect to the norm
$$
\bigl\|\mathsf u;
H^2_{0,\beta}\bigl((-\infty,R)\times\Omega\bigr)\bigr\|=\left(\sum_{\ell+|m|\leq
2}\int_{-\infty}^{R} \|e^{-\beta
x}\partial_x^\ell\partial_y^m\mathsf
u(x);L^2(\Omega)\|^2\,dx\right)^{1/2}.
$$
Let $L^2_\beta\bigl((-\infty,R)\times\Omega\bigr)$ be the weighted
$L^2$-space with the norm
$$
\bigl\|\mathsf f;
L^2_\beta\bigl((-\infty,R)\times\Omega\bigr)\bigr\|=\left(\int_{-\infty}^{R}\|e^{-\beta
x}\mathsf f(x); L^2(\Omega) \|^2\,dx \right)^{1/2}.
$$
Assume that $\lambda\in\mathcal D_\alpha\setminus\Bbb R$,
$\mu_0\in\Bbb R\setminus\sigma(\Delta_\Omega)$, and $\beta$ is in
the interval~\eqref{interval}. Then for any $\mathsf f\in
L^2_\beta\bigl((-\infty,R)\times\Omega\bigr)$ there exists a unique
 solution $\mathsf u\in H^2_{0,0}\bigl((-\infty,R)\times\Omega\bigr)$ to the
 equation
\begin{equation}\label{eq}
(\Delta_\Omega-(1+\lambda)^{-2}\partial^2_x -\mu_0)\mathsf u=\mathsf
f.
\end{equation}
Moreover, the estimate
\begin{equation}\label{eqest} \bigl\|\mathsf
u; H^2_{0,\beta}\bigl((-\infty,R)\times\Omega\bigr)\bigr\|\leq
C\bigl\|\mathsf f;
L^2_\beta\bigl((-\infty,R)\times\Omega\bigr)\bigr\|
\end{equation}
 holds, where the constant $C$ is independent of
$\mathsf f$ and $R$.
\end{lemma}
\begin{proof} It suffices to prove the assertion for $R=0$. Then the
general case can be obtained by the change of variables $x\mapsto
x-R$.

The set $C_c^\infty\bigl((-\infty,0)\times\Omega\bigr)$ of smooth
functions with compact supports  is dense in
$L^2_\beta\bigl((-\infty,0)\times\Omega\bigr)$. We first assume that
$\mathsf f\in C_c^\infty\bigl((-\infty,0)\times\Omega\bigr)$ and
extend $\mathsf f$ to a function in  $C_c^\infty(\Bbb
R\times\Omega)$ by setting $\mathsf f(-x)=-\mathsf f(x)$ for $x<0$.
Consider the equation~\eqref{eq} in the infinite cylinder $\Bbb
R\times\Omega$. As is known~\cite{Titchmarsh}, the Fourier transform
$\hat{\mathsf f}(\xi)=\int_{\Bbb R} e^{ix\xi}\mathsf f(x)\,dx$ is an
entire function of $\xi$ with values in $L^2(\Omega)$; it is rapidly
decaying at infinity in any strip $\{\xi\in\Bbb C: |\Im\xi|<\beta\}$
in the sense that the estimates $\|\hat{\mathsf f}(\xi);
L^2(\Omega)\|\leq C_{\beta,k}(1+|\xi|)^{-k}$ hold  for
$k=0,1,2,\dots$ and some constants $C_{\beta,k}$. Since $\beta$ is
in the interval~\eqref{interval}, the distance $d$ between the set
$\{\mu_0-(1+\lambda)^{-2}\xi^2:0\leq\Im\xi<\beta\}$ and the spectrum
$\sigma(\Delta_\Omega)$  of the selfadjoint operator $\Delta_\Omega$
in $L^2(\Omega)$ with domain $H^2_0(\Omega)$ is positive. Hence for
$$
\Psi(\xi)=\bigl(\Delta_\Omega+(1+\lambda)^{-2}\xi^2 -\mu_0\bigr)
^{-1}\hat{\mathsf f}(\xi),\quad 0\leq\Im\xi<\beta,
$$ we have the estimate $ \|\Psi(\xi);
L^2(\Omega)\|^2\leq d^{-2}\|\hat{\mathsf f}(\xi); L^2(\Omega)\|^2$.
This together with the  elliptic coercive estimate $$
 \|\Psi(\xi); H_0^{2}(\Omega)\|^2 \leq  c(\|\hat{\mathsf
f}(\xi); L^2(\Omega)\bigr\|^2 +\|\Psi(\xi); L^2(\Omega)\|^2)
$$
for the Dirichlet Laplacian  $\Delta_\Omega$ gives
 \begin{equation}\label{est0-}
 \|\Psi(\xi); H_0^{2}(\Omega)\|^2 \leq  (c+d^{-2})\|\hat{\mathsf f}(\xi); L^2(\Omega)\|^2,\quad 0\leq \Im\xi<\beta.
\end{equation}
The differential
operator~$\Delta_\Omega-(1+\lambda)^{-2}\partial^2_x$ is strongly
elliptic. Therefore the local coercive estimate
\begin{equation}\label{cest}
\|\varrho \mathsf u; H^2_0(\mathbb R\times\Omega)\|^2\leq
c\Bigl(\bigl\|\varsigma\mathsf f; L^2(\mathbb
R\times\Omega)\bigr\|^2 +\|\varsigma \mathsf u; L^2(\mathbb
R\times\Omega)\|^2\Bigr)
\end{equation}
is valid,  where $\varrho$ and $\varsigma$ are smooth functions of
the variable $x$ with compact supports such that $\varrho\not\equiv
0$ and $\varrho\varsigma=\varrho$.  We substitute $\mathsf u(x,
y)=e^{i\xi x}\Psi(\xi, y)$ into~\eqref{cest}. After simple
manipulations we arrive at the estimate
\begin{equation}\label{est0}
\sum_{\ell+|m|\leq 2}  |\xi|^{2\ell} \|\partial_y^m\Psi(\xi)
;L^2(\Omega)\|^2 \leq C\Bigl(\bigl\|\hat{\mathsf f}(\xi);
L^2(\Omega)\bigr\|^2+\|\Psi(\xi); L^2(\Omega)\|^2\Bigr),
\end{equation}
where the constant $C$ depends on $\varrho$ and $\varsigma$, but not
on $\xi$ or $\mathsf f$. If $|\xi|> C$ with sufficiently large
$C>0$, then the last term in \eqref{est0} can be neglected. This
together with~\eqref{est0-}
 justifies the estimate
\begin{equation}\label{estF}
\sum_{\ell+|m|\leq 2}|\xi|^{2\ell}  \|\partial_y^m \Psi(\xi);
L^2(\Omega)\|^2 \leq C \|\hat{\mathsf f}(\xi); L^2(\Omega)\|^2,
\end{equation}
 where $0\leq\Im\xi<\beta$ and the constant $C$ is independent of $\xi$ and
 $\Psi(\xi)$. Therefore the analytic in  strip $0\leq \Im\xi<\beta$ function $\xi\mapsto\Psi(\xi)\in
H^2_0(\Omega)$ is  rapidly decaying at infinity. This together with
the Cauchy integral theorem allows us to replace the contour of
integration in the inverse Fourier transformation $\mathsf
u(x)=(2\pi)^{-1}\int_{\Bbb R} e^{-i\xi x}\Psi(\xi)\,d\xi$. We obtain
$$
\begin{aligned}
\mathsf u(x)=& (2\pi)^{-1}\int_{\Bbb R} e^{-i\xi
x}\bigl(\Delta_\Omega+(1+\lambda)^{-2}\xi^2 -\mu_0\bigr)
^{-1}\hat{\mathsf f}(\xi)\,d\xi
\\
=&(2\pi)^{-1}\int_{\xi-i\beta\in\Bbb R} e^{-i\xi
x}\bigl(\Delta_\Omega+(1+\lambda)^{-2}\xi^2 -\mu_0\bigr)
^{-1}\hat{\mathsf f}(\xi)\,d\xi.
\end{aligned}
$$
The Parseval equality gives
$$
2\pi\int_\mathbb R \|e^{-\beta x}\partial_x^\ell\partial_y^m \mathsf
u(x); L^2(\Omega)\|^2\, dx= \int_{\xi-i\beta\in\mathbb R}
|\xi|^{2\ell} \|\partial_y^m\Psi(\xi); L^2(\Omega)\|^2\,d\xi,
$$
$$
2\pi\int_\mathbb R\|e^{-\beta x}\mathsf f(x); L^2(\Omega)
\|^2\,dx=\int_{\xi-i\beta\in\mathbb R} \|\hat{\mathsf f}(\xi);
L^2(\Omega)\|^2\,d\xi.
$$
Integrating~\eqref{estF} with respect to $\xi$, $\xi-i\beta\in\Bbb
R$, we deduce the estimate
\begin{equation}\label{finest}
 \sum_{\ell+|m|\leq 2}\int_{-\infty}^\infty  \|e^{-\beta
x}\partial_x^\ell\partial_y^m\mathsf u(x);L^2(\Omega)\|^2\,dx\leq C
\int_{-\infty}^\infty\|e^{-\beta x}\mathsf f(x); L^2(\Omega)
\|^2\,dx;
\end{equation}
in the case $\beta=0$ this estimate takes the form $\|\mathsf u;
H^2_0(\Bbb R\times\Omega)\|\leq C \|\mathsf f; L^2(\Bbb
R\times\Omega)\|$.  Thus for any $\mathsf f\in C_c^\infty(\Bbb
R\times\Omega)$ there exists a solution $\mathsf u\in H^2_0(\Bbb
R\times\Omega)$ to the equation~\eqref{eq} and the
estimate~\eqref{finest} holds with any $\beta$ in the
interval~\eqref{interval}. Usual argument on smoothness of solutions
to elliptic problems gives $\mathsf u\in C^\infty(\Bbb
R\times\overline{\Omega})$. From the equality $\mathsf
f(-x)=-\mathsf f(x)$ it follows that $\hat{\mathsf
f}(-\xi)=-\hat{\mathsf f}(\xi)$ and therefore
$\Psi(\xi)=-\Psi(-\xi)$. Hence $\mathsf u(x)=-\mathsf u(-x)$ and
$\mathsf u(0)=0$. As in the proof of Proposition~\ref{sectorial} we
conclude that in the norm of $H^2_{0,\beta}
\bigl((0,\infty)\times\Omega\bigr)$ one can approximate $\mathsf u$
by functions in $C_0^\infty \bigl((0,\infty)\times\Omega\bigr)$.
Hence $\mathsf u\in H^2_{0,\beta}
\bigl((-\infty,0)\times\Omega\bigr)$.  By continuity our
construction extends to all $\mathsf f\in
L^2_\beta\bigl((-\infty,0)\times\Omega\bigr)$. In particular, for
any $f\in L^2_0\bigl((-\infty,0)\times\Omega\bigr)$ we can find a
solution $\mathsf u\in H^2_{0,0}\bigl((-\infty,0)\times\Omega\bigr)$
to the equation~\eqref{eq}. If $\mathsf f\in L^2_\beta
\bigl((-\infty,0)\times\Omega\bigr)$ with some $\beta$ in the
interval~\eqref{interval}, then the estimate~\eqref{eqest} with
$R=0$ is a direct consequence of~\eqref{finest}. It remains to note
that $\mathsf u\in H^2_{0,0}\bigl((-\infty,0)\times\Omega\bigr)$ is
a unique solution as our argument also shows that for any $\mathsf
f\in L^2_0\bigl((-\infty,0)\times\Omega\bigr)$ the adjoint equation
$(\Delta_\Omega-(1+\overline{\lambda})^{-2}\partial^2_x
-\mu_0)\mathsf u=\mathsf f$ is solvable in the space
$H^2_{0,0}\bigl((-\infty,0)\times\Omega\bigr)$.
\end{proof}

 Now we are in position
to prove Theorem~\ref{stability}.
\begin{proof} The scheme of the proof is similar to the one  we used in
the proof of~\cite[Theorem 4.1]{KalvinSiNum}. We rely on a
modification of the compound expansion method~\cite{ref6}. We say
that $w=w(R)\in H^2_0(\mathcal G_R)$ is an approximate solution of
the problem with finite PML if the following conditions are
satisfied:
\begin{itemize}
\item[i.] The estimate $\|w; H^2_0(\mathcal G_R)\| \leq
c\|g;L^2(\mathcal G_R)\|$ holds with an independent of $g$ and $R$
constant $c$;
\item[ii.] The estimate $\|(\Delta_{\lambda,r}-\mu_0)w-g; L^2(\mathcal G_R)\|\leq C_R\|g; L^2(\mathcal G_R)\|$
is valid, where the constant $C_R$ is independent of $g$ and $C_R\to
0$ as $R\to+\infty$.
\end{itemize}
Due to condition i $w$ continuously depends on $g$.   Condition ii
implies that the discrepancy, left by $w$ in the
equation~\eqref{PML}, tends to zero as $R\to+\infty$. Once an
approximate solution $w$ is found, it is not hard to verify the
assertion of the theorem.

Let $\rho\in C^\infty (\Bbb R)$ be a cutoff function such that
$\rho(x)=1$ for $x\leq 0$ and $\chi(x)=0$ for $x\geq 1/2$. We set
$\rho_R=\rho(x-R)$,
$\varrho_R(\zeta,\eta)=\rho_R\circ\varkappa^{-1}(\zeta,\eta)$ for
$(\zeta,\eta)\in\overline{ \mathcal C}$, and
$\varrho_R(\zeta,\eta)=1$ for $(\zeta,\eta)\in\overline{\mathcal
G}\setminus\overline{\mathcal C}$. Let $\mathcal F=\varrho_{R/2} g$
and $\mathsf f=(g-f)\upharpoonright_{\mathcal G_R}\circ\varkappa$.
We extend $\mathcal F$ from $\mathcal G_R$ to $\mathcal G$ and
$\mathsf f$ from $(0,R)\times\Omega$ to $(-\infty,R)\times\Omega$ by
zero. We find an approximate solution $w$ compounded of
$u_{\lambda,r}=(\Delta_{\lambda,r}-\mu_0)^{-1}\mathcal F$ and a
solution $\mathsf u\in H^2_{0,0}\bigl((-\infty,R)\times\Omega\bigr)$
to the equation~\eqref{eq} in the form
$$
w=\varrho_R u_{\lambda,r}+(1-\varrho_{R/3})(\mathsf
u\circ\varkappa^{-1});
$$
here the second term in the right hand side is extended from
$\mathcal G_R\cap\mathcal C$ to $\mathcal G_R$ by zero.

Let us show that $w$ is an approximate solution. Observe that on the
support of $f$ we have $e^{\beta\mathsf s}\leq C e^{\beta R/2}$ and
on the support of $\mathsf f$ we have $e^{-\beta x}\leq C e^{-\beta
R/2}$ uniformly in $R$. Hence
\begin{equation}\label{rhs}
\|e^{\beta\mathsf s}\mathcal F; L^2(\mathcal G)\|+e^{\beta
R}\bigl\|\mathsf f;
L^2_\beta\bigl((-\infty,R)\times\Omega\bigr)\bigr\|\leq C e^{\beta
R/2}\| g; L^2(\mathcal G)\|
\end{equation}
with an independent of $R$ and $g$ constant $C$. Similarly
to~\eqref{www} we conclude that
$$
\begin{aligned}
\|w; H^2_0(\mathcal G_R)\|^2\leq \|\varrho_R
u_{\lambda,r};H^2_0(\mathcal
G)\|^2+c\Bigl(\bigl\|\triangle_{\lambda,r}(1-\rho_{R/3})\mathsf
u;L^2\bigl((-\infty,R)\times\Omega\bigr)\bigr\|\\ +\bigl\|\mathsf
u;L^2\bigl((-\infty,R)\times{\Omega}\bigr)\bigr\|\Bigr)^2 \leq
C\|u_{\lambda,r};H^2_0(\mathcal
G)\|^2+C\bigl\|u;H^2_{0,0}\bigl((-\infty,0)\times\Omega\bigr)\bigr\|^2,
\end{aligned}
$$
where $C$ is independent of $R$ and $\triangle_{\lambda,r}$ is the
operator~\eqref{LB_l}. This together with
 the estimates~\eqref{eqest},~\eqref{rhs} for
$\beta=0$ and Theorem~\ref{LA}.1 implies that the condition i is
satisfied.

Let us verify the condition ii. We have
\begin{equation}\label{9.5}
\begin{aligned}
(\Delta_{\lambda,r}-\mu_0)w-g=[\Delta_{\lambda,r},\varrho_R]u_{\lambda,r}+(1+\lambda)^{-2}\bigl([\partial^2_x,\rho_{R/3}]\mathsf
u\bigr)\circ\varkappa^{-1}
\\+\Bigl(\bigl(\triangle_{\lambda,r}-\Delta_\Omega+(1+\lambda)^{-2}\partial^2_x\bigr)(1-\rho_{R/3})\mathsf
u\Bigr)\circ\varkappa^{-1}.
\end{aligned}
\end{equation}
The support of the term
$[\Delta_{\lambda,r},\varrho_R]u_{\lambda,r}$ is a subset of the
image of $(R,R+1/2)\times\Omega$ under the diffeomorphism
$\varkappa$. On this support the weight $e^{\beta\mathsf s}$ is
bounded from below by $c e^{\beta R}$ uniformly in $R>0$. As a
consequence we get the uniform in $R$ estimates
\begin{equation}\label{9.6}
\|[\Delta_{\lambda,r},\varrho_R]u_{\lambda,r}; L^2(\mathcal
G_R)\|\leq C_1 e^{-\beta R}\|e^{\beta\mathsf
s}u_{\lambda,r};H^2_0(\mathcal G)\|\leq C_2 e^{-\beta
R}\|e^{\beta\mathsf s}\mathcal F; L^2(\mathcal G)\|,
\end{equation}
where we used Theorem~\ref{ED}. Now we estimate the second term in
the right hand side of~\eqref{9.5}. On the support of
$[\partial^2_x,\rho_{R/3}]\mathsf u$ we have $e^{-\beta x}\geq C
e^{-\beta R/3}$. Relying on~\eqref{eqest} we obtain
\begin{equation}\label{9.7}
\begin{aligned}
\bigl\|&(1+\lambda)^{-2}\bigl([\partial^2_x,\rho_{R/3}]\mathsf
u\bigr)\circ\varkappa^{-1}; L^2(\mathcal G_R)\bigr\|\leq
c\bigl\|[\partial^2_x,\rho_{R/3}]\mathsf u;
L^2\bigl((-\infty,R)\times\Omega\bigr)\bigr\|
\\
&\leq C_1 e^{\beta R/3} \bigl\|
u;H^2_{0,\beta}\bigl((-\infty,R)\times\Omega\bigr)\bigr\| \leq
C_2e^{\beta R/3}\bigl\|\mathsf f;
L^2\bigl((-\infty,R)\times\Omega\bigr) \bigr\|.
\end{aligned}
\end{equation}
Finally, consider the last term in the right hand side
of~\eqref{9.5}. On the support of $(1-\rho_{R/3})\mathsf u$ the
coefficients of the operator
$\triangle_{\lambda,r}-\Delta_\Omega+(1+\lambda)^{-2}\partial^2_x$
tend to zero as $R\to +\infty$, cf.~\eqref{lim} and~\eqref{LB_l}.
This together with the estimate~\eqref{eqest} for $\beta=0$ gives
\begin{equation}\label{9.8}
\begin{aligned}
\Bigl\| &
\Bigl(\bigl(\triangle_{\lambda,r}-\Delta_\Omega+(1+\lambda)^{-2}\partial^2_x\bigr)(1-\rho_{R/3})\mathsf
u\Bigr)\circ\varkappa^{-1}; L^2(\mathcal G_R)\Bigr\|
\\
& \leq c_R\bigl\|\mathsf f;
L^2_0\bigl((-\infty,R)\times\Omega\bigr)\bigr\|,
\end{aligned}
\end{equation}
where $c_R\to 0$ as $R\to+\infty$. From~\eqref{rhs}--\eqref{9.8} it
follows that $w$ meets  the condition ii. Thus $w=w(R)$ is indeed an
approximate solution to the problem with finite PML.

Now we are in position to prove the assertion of the theorem.
Observe that $(\Delta_{\lambda,r}-\mu_0)w-g=\mathfrak O(R) g$ with
some operator $\mathfrak O(R)$ in $L^2(\mathcal G_R)$, whose norm
$|\mspace{-2mu}|\mspace{-2mu}|\mathfrak
O(R)|\mspace{-2mu}|\mspace{-2mu}|$ tends to zero as $R\to+\infty$
because of the condition ii on $w$. For all $R>R_0$ with a
sufficiently large $R_0$ we have
$|\mspace{-2mu}|\mspace{-2mu}|\mathfrak
O(R)|\mspace{-2mu}|\mspace{-2mu}|\leq|\mspace{-2mu}|\mspace{-2mu}|\mathfrak
O(R_0)|\mspace{-2mu}|\mspace{-2mu}|<1$. Hence there exists the
inverse $(I+\mathfrak O(R))^{-1}:L^2(\mathcal G_R)\to L^2(\mathcal
G_R)$ and its norm is bounded by the constant
$1/(1-|\mspace{-2mu}|\mspace{-2mu}|\mathfrak
O(R_0)|\mspace{-2mu}|\mspace{-2mu}|)$ uniformly in $R>R_0$. We set
$\tilde g=(I+\mathfrak O(R))^{-1}g$. In the same way as before we
construct the approximate solution $w$ for the problem~\eqref{PML},
where $g$ is replaced by $\tilde g$. Then for $v=w$ we have
$(\Delta_{\lambda,r}-\mu_0)v=\tilde g+\mathfrak O(R) \tilde g=g$ and
$$
 \|v;H^2_0(\mathcal
G_R)\|\leq c\|\tilde g; L^2(\mathcal G_R)\|\leq
c/(1-|\mspace{-2mu}|\mspace{-2mu}|\mathfrak
O(R_0)|\mspace{-2mu}|\mspace{-2mu}|)\|g; L^2(\mathcal G_R)\|,$$
where $C$ is independent of $R>R_0$. Thus for $R>R_0$ and $g\in
L^2(\mathcal G_R)$ there exists  a solution $v\in H^2_0(\mathcal
G_R)$ to the
 equation~\eqref{PML} satisfying the estimate~\eqref{e+}, where the
constant $C$ is independent of $R$ and $g$. In the remaining part of
the proof we show that this solution is unique.

Let $\Delta_{\lambda,r}^R$ be the unbounded operator in
$L^2(\mathcal G_R)$ such that for any $v$ in its domain $
H^2_0(\mathcal G_R)$ we have  $\Delta_{\lambda,r}^R
v=\Delta_{\lambda,r}v$. Note that  $\Delta_{\lambda, r}^R$ is the
operator of the problem with finite PML.  To the operator
$\Delta_{\lambda, r}^R$ there corresponds the quadratic form
$$
q_{\lambda,r}^R[v,v]=\int_{\mathcal G_R} \bigl\langle\bigl({\det
\mathsf e_{\lambda,r} }\bigr)^{1/2}\mathsf e_{\lambda,r}^{-1}
\nabla_{\zeta \eta}v, \nabla_{\zeta \eta}\bigl({\det \mathsf
e_{\overline{\lambda},r} }\bigr)^{-1/2} v\bigr\rangle\,d\zeta\,d\eta
$$
in $L^2(\mathcal G_R)$ with the domain $H^2_0(\mathcal G_R)$. In the
same way as in Section~\ref{secAF} we conclude that the form
$q_{\lambda,r}^R$ admits a densely defined sectorial closure with
 the domain
$\lefteqn{\stackrel{\circ}{\phantom{\,\,>}}}H^1(\mathcal G_R)$,
where $\lefteqn{\stackrel{\circ}{\phantom{\,\,>}}}H^1(\mathcal G_R)$
is the completion of  $H^2_0(\mathcal G_R)$ with respect to the norm
$\|v; \lefteqn{\stackrel{\circ}{\phantom{\,\,>}}}H^1(\mathcal
G_R)\|=(\sum_{\ell+|m|\leq 1}\|\partial_\zeta^\ell\partial_\eta^m
v;L^2(\mathcal G_R)\|^2)^{1/2}$. This together with the argument
above implies that all $\mu<0$ with sufficiently large absolute
value are regular points of the operator $\Delta_{\lambda, r}^R$.
Hence the sectorial operator $\Delta_{\lambda, r}^R$ coincides with
its m-sectorial Friedrichs extension. Moreover, thanks to the
argument above we know that under the assumptions of theorem for any
$g\in L^2(\mathcal G_R)$ there exists $v\in H^2_0(\mathcal G_R)$
such that $(\Delta_{\lambda,r}^R-\mu_0)v=g$. Similarly, one can
study the adjoint m-sectorial operator $(\Delta_{\lambda, r}^R)^*$.
It turns out that under the  assumptions of theorem for any $g\in
L^2(\mathcal G_R)$ there exists $v$ in the domain $H^2_0(\mathcal
G_R)$ of $(\Delta_{\lambda, r}^R)^*$ such that
$\bigl((\Delta^R_{\lambda,r})^*-\mu_0\bigr)v=g$. Therefore the
equation~\eqref{PML} is uniquely solvable in $H^2_0(\mathcal G_R)$.
\end{proof}

In the next theorem we show that under some natural assumptions
solutions $v=v(R)$ of the problem with finite PMLs converge in the
domain $\mathcal G_r$ to outgoing or incoming solutions $u_\pm$ with
an exponential rate as $R\to+\infty$. In other words, we estimate
the error produced by truncation of infinite PMLs. Solutions to the
problem with finite PMLs can be found numerically with the help of
finite element solvers; certainly, discretization produces yet
another error that we do not estimate here.
\begin{theorem}\label{convergence}  Assume that $\mu_0\in\Bbb
R\setminus\sigma(\Delta_\Omega)$ is not an eigenvalue of the
selfadjoint Dirichlet Laplacian $\Delta$ in $L^2(\mathcal G)$, the
parameter $r>0$ is sufficiently large in the sense of
Remark~\ref{rem r}, $\lambda\in\mathcal D_\alpha\setminus\Bbb R$,
and $\beta$ is in the interval~\eqref{interval}. Let $f\in\mathcal
H_\alpha(\mathcal G)$ satisfy the inclusion $e^{\beta \mathsf
s}(f\circ\vartheta_{\lambda,r})\in L^2(\mathcal G)$, where $\mathsf
s$ is the same function as in Theorem~\ref{ED}. In~\eqref{PML} we
set $g=f\circ\vartheta_{\lambda,r}$. Then there exists $R_0>r$ such
that for $R>R_0$ a unique solution $v=v(R)\in H^2_0(\mathcal G_R)$
of the problem with finite PMLs
 converges in  $\mathcal
G_r$
\begin{enumerate}
\item to the outgoing solution
$u_-\in H^2_{0,\operatorname{loc}}(\mathcal G)$ of the equation
$(\Delta-\mu_0)u=f$ in the case $\Im\lambda>0$
\item to the incoming solution
$u_+\in H^2_{0,\operatorname{loc}}(\mathcal G)$ of the equation
$(\Delta-\mu_0)u=f$ in the case $\Im\lambda<0$
\end{enumerate}
in the  sense that as $R\to+\infty$ the estimate
\begin{equation}\label{fin}
\sum_{\ell+|m|\leq
2}\|\partial_\zeta^\ell\partial_\eta^m(u_\pm-v_R); L^2(\mathcal
G_r)\|^2\leq Ce^{-2\beta R}\|e^{\beta\mathsf
s}(f\circ\vartheta_{\lambda,r}); L^2(\mathcal G)\|^2
\end{equation}
holds with a constant $C$ independent of $R>R_0$ and $f$.
\end{theorem}

Let us remark here that the assumptions of Theorem~\ref{convergence}
on the right hand side $f$ are a priori met for all $f\in
L^2(\mathcal G)$ such that $f\upharpoonright\mathcal
C=F\circ\varkappa$ with some $F\in\mathscr E$; here $\mathscr E$ is
the algebra defined in Section~\ref{s3} and $\varkappa$ is the
diffeomorphism~\eqref{diff}. From Lemma~\ref{p1} it follows that the
set of functions $f$ satisfying the assumptions of
Theorem~\ref{convergence} is dense in $L^2(\mathcal G)$. In
particular, we can take any $f\in L^2(\mathcal G)$ supported in
$\mathcal G_r$.
\begin{proof}
Let $\rho\in C^\infty (\Bbb R)$ be a cutoff function such that
$\rho(x)=1$ for $x\leq 0$ and $\chi(x)=0$ for $x\geq 1/2$. We set
$\rho_R=\rho(x-R)$,
$\varrho_R(\zeta,\eta)=\rho_R\circ\varkappa^{-1}(\zeta,\eta)$ for
$(\zeta,\eta)\in \overline{\mathcal C}$, and
$\varrho_R(\zeta,\eta)=1$ for $(\zeta,\eta)\in\overline{\mathcal
G}\setminus\overline{\mathcal C}$. Thanks to Theorem~\ref{LA}.3 it
suffices to prove the estimate~\eqref{fin} with $u_\pm$ replaced by
$\varrho_R u_{\lambda,r}$. The difference $\varrho_R
u_{\lambda,r}-v_R\in H^2_0(\mathcal G_R)$ satisfies the
problem~\eqref{PML} with
$g=(\varrho_R-1)(f\circ\vartheta_{\lambda,r})+[\Delta_{\lambda,r},\varrho_R]u_{\lambda,r}$.
Observe that
$$
\begin{aligned}
\|(\varrho_R-1)(f\circ\vartheta_{\lambda,r}); L^2(\mathcal
G_R)\|\leq C e^{-\beta R}\|e^{\beta\mathsf
s}(f\circ\vartheta_{\lambda,r}); L^2(\mathcal G)\|,
\\
 \|[\Delta_{\lambda,r},\varrho_R]u_{\lambda,r};L^2(\mathcal
G_R)\|\leq C e^{-\beta R}\|e^{\beta \mathsf s} u_{\lambda,r};
H^2_0(\mathcal G)\|,
\end{aligned}
$$
because the functions $(\varrho_R-1)(f\circ\vartheta_{\lambda,r})$
and $[\Delta_{\lambda,r},\varrho_R]u_{\lambda,r}$, being written in
the coordinates $(x,y)$, are equal to zero for $x<R$, while
$e^{\beta \mathsf s(\zeta,\eta)}=c e^{\beta x}$ for $x\geq R>C$,
cf.~\eqref{ab3}. This together with Theorem~\ref{ED} gives
\begin{equation}\label{fin-2}
\|g; L^2(\mathcal G_R)\|\leq c e^{-\beta R}\|e^{\beta\mathsf
s}(f\circ\vartheta_{\lambda,r}); L^2(\mathcal G)\|.
\end{equation}
By Theorem~\ref{stability} we have
\begin{equation}\label{fin-1}
\|\varrho_R u_{\lambda,r}-v_R; H^2_0(\mathcal G_R)\|\leq C \|g;
L^2(\mathcal G_R)\|,\quad R>R_0.
\end{equation}
It remains to note that
$$
\sum_{\ell+|m|\leq
2}\|\partial_\zeta^\ell\partial_\eta^m(u_\pm-v_R); L^2(\mathcal
G_r)\|^2\leq \|\varrho_R u_{\lambda,r}-v_R; H^2_0(\mathcal
G_R)\|^2,\quad R>R_0>r.
$$
This together with~\eqref{fin-1} and~\eqref{fin-2} completes the
proof of the estimate~\eqref{fin}.
\end{proof}
\begin{remark} By Theorem~\ref{convergence}
 the rate of convergence of the PML method depends only on the spectral parameter $\mu_0$ and the infinitely distant
cross-section $\Omega$. In the case $\mathcal
C=(0,\infty)\times\Omega$ the quasi-cylindrical domain $\mathcal G$
corresponds to a resonator with attached tubular waveguide. In this
particular case the results of Theorem~\ref{convergence} for $f\in
L^2(\mathcal G)$ with $\supp f\subset \mathcal G_r$ can equivalently
be obtained by the modal expansions technique,
e.g.~\cite{ref4+,ref5,ref5+}.
\end{remark}

\end{document}